\documentclass[a4paper,UKenglish,cleveref, autoref, thm-restate]{lipics-v2021}
\usepackage{extarrows}
\usepackage{graphicx}
\usepackage{mathtools}
\usepackage[ruled]{algorithm2e}
\usepackage{tikz}
\usetikzlibrary{shapes,snakes,backgrounds,arrows,patterns}
\usetikzlibrary {patterns.meta}
\pgfdeclarepattern{
    name=hatch,
    parameters={\hatchsize,\hatchangle,\hatchlinewidth},
    bottom left={\pgfpoint{-.1pt}{-.1pt}},
    top right={\pgfpoint{\hatchsize+.1pt}{\hatchsize+.1pt}},
    tile size={\pgfpoint{\hatchsize}{\hatchsize}},
    tile transformation={\pgftransformrotate{\hatchangle}},
    code={
        \pgfsetlinewidth{\hatchlinewidth}
        \pgfpathmoveto{\pgfpoint{-.1pt}{-.1pt}}
        \pgfpathlineto{\pgfpoint{\hatchsize+.1pt}{\hatchsize+.1pt}}
        \pgfusepath{stroke}
    }
}

\tikzset{
    hatch size/.store in=\hatchsize,
    hatch angle/.store in=\hatchangle,
    hatch line width/.store in=\hatchlinewidth,
    hatch size=5pt,
    hatch angle=0pt,
    hatch line width=.5pt,
}

\allowdisplaybreaks
\nolinenumbers

\newtheorem{lem}[theorem]{Lemma}
\newtheorem{prop}[theorem]{Proposition}

\theoremstyle{definition}
\newtheorem{defn}[theorem]{Definition}
\theoremstyle{definition}
\newtheorem{problem}[theorem]{Problem}
\theoremstyle{definition}
\newtheorem{obs}[theorem]{Observation}
\theoremstyle{definition}
\newtheorem{rmk}[theorem]{Remark}
\theoremstyle{definition}

\newcommand{\Z}{\mathbb{Z}}
\newcommand{\N}{\mathbb{N}}
\newcommand{\Q}{\mathbb{Q}}

\newcommand{\K}{\mathbb{K}}

\newcommand{\Rp}{\mathbb{R}_{\geq 0}}

\newcommand{\lcm}{\operatorname{lcm}}
\newcommand{\mG}{\mathcal{G}}

\newcommand{\mI}{\mathcal{I}}
\newcommand{\tI}{\widetilde{\mathcal{I}}}

\newcommand{\tJ}{\widetilde{\mathcal{J}}}

\newcommand{\mH}{\mathcal{H}}
\newcommand{\mA}{\mathcal{A}}
\newcommand{\mM}{\mathcal{M}}
\newcommand{\mZ}{\mathcal{Z}}

\newcommand{\ba}{\boldsymbol{a}}
\newcommand{\bb}{\boldsymbol{b}}
\newcommand{\bc}{\boldsymbol{c}}
\newcommand{\bm}{\boldsymbol{m}}

\newcommand{\bzer}{\boldsymbol{0}}

\newcommand{\gen}[1]{\langle {#1} \rangle}
\newcommand{\bff}{\boldsymbol{f}}
\newcommand{\bg}{\boldsymbol{g}}

\newcounter{ProblemCounter}
\newcounter{DecideCounter}

\title{Subgroup and Coset Intersection in abelian-by-cyclic groups}
\author{Ruiwen Dong}{Department of Computer Science, University of Oxford}{ruiwen.dong@kellogg.ox.ac.uk}{}{}
\authorrunning{R. Dong}

\Copyright{Ruiwen Dong} %TODO mandatory, please use full first names. LIPIcs license is "CC-BY";  http://creativecommons.org/licenses/by/3.0/

 %TODO mandatory: Please choose ACM 2012 classifications from https://dl.acm.org/ccs/ccs_flat.cfm 

 %TODO mandatory; please add comma-separated list of keywords

\category{} %optional, e.g. invited paper

\relatedversion{} %optional, e.g. full version hosted on arXiv, HAL, or other respository/website
\supplement{}%optional, e.g. related research data, source code, ... hosted on a repository like zenodo, figshare, GitHub, ...

\begin{document}

\maketitle
\begin{abstract}
We consider two decision problems in infinite groups.
The first problem is Subgroup Intersection: given two finitely generated subgroups $\langle \mathcal{G} \rangle, \langle \mathcal{H} \rangle$ of a group $G$, decide whether the intersection $\langle \mathcal{G} \rangle \cap \langle \mathcal{H} \rangle$ is trivial.
The second problem is Coset Intersection: given two finitely generated subgroups $\langle \mathcal{G} \rangle, \langle \mathcal{H} \rangle$ of a group $G$, as well as elements $g, h \in G$, decide whether the intersection of the two cosets $g \langle \mathcal{G} \rangle \cap h \langle \mathcal{H} \rangle$ is empty.
We show that both problems are decidable in finitely generated abelian-by-cyclic groups.
In particular, we reduce them to the Shifted Monomial Membership problem (whether an ideal of the Laurent polynomial ring over integers contains any element of the form $X^z - f,\; z \in \mathbb{Z} \setminus \{0\}$). 
We also point out some obstacles for generalizing these results from abelian-by-cyclic groups to arbitrary metabelian groups.
\end{abstract}

\newpage

\section{Introduction}
\paragraph*{Algorithmic problems in groups}
Computational group theory is one of the
oldest and most well-developed parts of computational algebra.
Dating back to the first half of the twentieth century, the area provided some of the first undecidability results in the theory of computing.
Among the most prominent problems is the Subgroup Membership problem, proposed by Mikhailova~\cite{mikhailova1966occurrence} in the 1960s.
For this problem, we work in a group $G$ that is typically infinite but finitely generated (such as the additive group of $\Z$, or a matrix group over integers).
For a finite subset $\mG$ of $G$, denote by $\gen{\mG}$ the subgroup of $G$ generated by $\mG$.
The Subgroup Membership problem for the group $G$ is defined as follows.
\begin{enumerate}[(i)]
    \item \textit{(Subgroup Membership)} Given a finite set of elements $\mG \subseteq G$ and $h \in G$, decide whether $h \in \gen{\mG}$.
    \setcounter{ProblemCounter}{\value{enumi}}
\end{enumerate}
Denote by $e$ the neutral element of $G$.
Another widely studied problem is the Subgroup Intersection problem:
\begin{enumerate}[(i)]
    \setcounter{enumi}{\value{ProblemCounter}}
    \item \textit{(Subgroup Intersection)} Given two finite sets of elements $\mG, \mH \subseteq G$, decide whether $\gen{\mG} \cap \gen{\mH}$ is the trivial group $\{e\}$.
    \setcounter{ProblemCounter}{\value{enumi}}
\end{enumerate}
While the intersection $\gen{\mG} \cap \gen{\mH}$ is always a group, it is usually not clear how to compute its effective representation. 
In fact, the group $G$ might not satisfy the \emph{Howson property}, meaning the intersection of two finitely generated subgroups $\gen{\mG}, \gen{\mH}$ of $G$ might not be finitely generated~\cite{Moldavanskii1968IntersectionOF}.
Deciding triviality of $\gen{\mG} \cap \gen{\mH}$ can be considered as a first step towards understanding intersection of subgroups in the given group $G$.

For an element $g \in G$ and a subgroup $S \leq G$, define the coset $g S \coloneqq \{gs \mid s \in S\}$.
We also consider the Coset Intersection problem:
\begin{enumerate}[(i)]
    \setcounter{enumi}{\value{ProblemCounter}}
    \item \textit{(Coset Intersection)} Given two finite sets of elements $\mG, \mH \subseteq G$ and $g, h \in G$, decide whether $g \gen{\mG} \cap h \gen{\mH}$ is empty.
    \setcounter{ProblemCounter}{\value{enumi}}
\end{enumerate}
Coset Intersection as well as Subgroup Intersection has been extensively studied in various contexts such as permutation groups~\cite{babai2010coset}, abelian and nilpotent groups~\cite{babai1996multiplicative, macdonald2019low}, and right-angled Artin groups~\cite{delgado2018intersection}.
They are intimately related to problems from numerous other areas such as Graph Isomorphism~\cite{luks1982isomorphism}, vector reachability~\cite{potapov2019vector} and automata theory~\cite{delgado2018intersection}.
Since the intersection $g \gen{\mG} \cap h \gen{\mH}$ is empty if and only if $\gen{\mG} \cap g^{-1} h \gen{\mH}$ is empty, we can without loss of generality suppose $g = e$ in the definition of Coset Intersection.
That is, we want to decide whether $\gen{\mG} \cap h \gen{\mH} = \emptyset$.
We may note that by setting $\mH = \{e\}$, Coset Intersection subsumes Subgroup Membership. However, Coset Intersection does not subsume Subgroup Intersection, despite their obvious connection.

It is not surprising that for general groups, all three problems are undecidable.
A classic result of Mikhailova~\cite{mikhailova1966occurrence} shows that Subgroup Membership is undecidable for the direct product $F_2 \times F_2$ of two free groups $F_2$ over two generators.
Moreover, Mikhailova's construction also implies that Subgroup Intersection is undecidable for $F_2 \times F_2$~\cite{199499}.

Nevertheless, for finitely generated free groups and abelian groups, Subgroup Membership, Subgroup Intersection and Coset Intersection have been shown to be decidable.
For free groups, these decidability results were obtained using the classic construction of \emph{Stallings foldings}~\cite[Proposition~6.1]{bogopolski2010orbit}~\cite[Proposition~7.2, Corollary~9.5]{KAPOVICH2002608}~\cite{stallings1991foldings}.
This automata-inspired construction has now become the standard tool for describing subgroups of free groups.
For abelian matrix groups, Babai, Beals, Cai, Ivanyos and Luks~\cite{babai1996multiplicative} famously reduced computational problems for commutative matrices to computation over lattices.
Therefore Subgroup Membership, Subgroup Intersection and Coset Intersection reduce to linear algebra over $\Z$ and are decidable in polynomial time.

\paragraph*{Metabelian and abelian-by-cyclic groups}
As most algorithmic problems for abelian groups are well-understood due to their relatively simple structure, much effort has focused on relaxations of the commutativity requirement.
For example, the aforementioned decidability results have been successfully extended to the class of nilpotent groups~\cite{macdonald2019low}.
Among the simplest and most well-studied extensions to abelian groups is the class of \emph{metabelian groups}.
A group $G$ is called metabelian if it admits an abelian normal subgroup $A$ such that the quotient group $G/A$ is abelian.
Developing a complete algorithmic theory for finitely generated metabelian groups has been the focus of intense research since the 1950s~\cite{baumslag1994algorithmic, hall1954finiteness}.

Despite their simple definition, many problems in finitely generated metabelian groups are still far from being well understood.
Unlike free groups, abelian groups and nilpotent groups, metabelian groups do not satisfy the Howson property~\cite{baumslag2010subgroups, howson1954intersection}.
This makes solving intersection-type problems in metabelian groups much more difficult.
Among the three problems introduced above, only the decidability of Subgroup Membership is known, thanks to a classic result of Romanovskii~\cite{romanovskii1974some}.
Subgroup Intersection has been solved only for \emph{free} metabelian groups~\cite{baumslag2010subgroups} and the wreath products $\Z^m \wr \Z^n, m, n \geq 1$.
Unfortunately, this solution does not generalize to arbitrary metabelian groups, as explicitly stated after~\cite[Corollary~C]{baumslag2010subgroups}.
Despite Subgroup Intersection and Coset Intersection being currently out of reach for arbitrary metabelian groups, various results have been obtained for specific classes of metabelian groups.
Recent results by Lohrey, Steinberg and Zetzsche~\cite{lohrey2015rational} showed decidability of the \emph{Rational Subset Membership} problem (which subsumes Coset Intersection) in the wreath products $(\Z/p\Z) \wr \Z,\; p \geq 2$.
This result has been extended to the \emph{Baumslag-Solitar groups} $\mathsf{BS}(1, p),\; p \geq 2,$ by Cadilhac, Chistikov and Zetzsche~\cite{DBLP:conf/icalp/CadilhacCZ20}.
The groups $(\Z/p\Z) \wr \Z$ and $\mathsf{BS}(1, p)$ can be respectively represented as groups of $2 \times 2$ matrices over the Laurent polynomial ring $\left(\Z / p \Z\right)[X, X^{-1}]$ and over the ring $\Z[1/p] = \{\frac{a}{p^n} \mid a \in \Z, n \in \N\}$:
\begin{align}
\left(\Z / p \Z\right) \wr \Z & \cong \left\{ 
\begin{pmatrix}
        X^{b} & f \\
        0 & 1
\end{pmatrix}
\;\middle|\; f \in \left(\Z / p \Z\right)[X, X^{-1}], b \in \Z 
\right\}, \label{eq:defwr}
\\
\mathsf{BS}(1, p) & \cong \left\{ 
\begin{pmatrix}
        p^{b} & f \\
        0 & 1
\end{pmatrix}
\;\middle|\; f \in \Z[1/p], b \in \Z 
\right\}. \label{eq:defbs}
\end{align}
Alternatively, the element 
$
\begin{pmatrix}
        X^{b} & f \\
        0 & 1
\end{pmatrix}
\in 
(\Z/p\Z) \wr \Z
$
can be thought of as a Turing machine configuration whose tape cells contain letters in $\{0, 1, \ldots, p-1\}$ which correspond to the coefficients of the polynomial $f$, while the head of the machine is positioned at the cell $b$.
Multiplication in $(\Z/p\Z) \wr \Z$ corresponds to operating the machine by moving the head and adding integers to the cells modulo $p$. (See~\cite{lohrey2015rational} for a complete description.)
Similarly, $\mathsf{BS}(1, p)$ can be considered as a version of $(\Z/p\Z) \wr \Z$ with ``carrying''.
The element
$
\begin{pmatrix}
        p^{b} & f \\
        0 & 1
\end{pmatrix}
\in 
\mathsf{BS}(1, p)
$
can be seen the base-$p$ expansion of the rational number $f \in \Z[1/p]$, along with a cursor at the $b$-th position. Multiplication in $\mathsf{BS}(1, p)$ corresponds to aligning the cursors of the two elements and adding up the numbers $f$. (See~\cite{DBLP:conf/icalp/CadilhacCZ20} for a complete description.)

The Turing machine-like structure of $(\Z/p\Z) \wr \Z$ and $\mathsf{BS}(1, p)$ can be explained by the following fact.
Both groups belong to the much broader class of groups called \emph{abelian-by-cyclic} groups.
A group is called abelian-by-cyclic if it admits an abelian normal subgroup $A$ such that the quotient group $G/A$ is isomorphic to $\Z$.
Intuitively, this isomorphism to $\Z$ gives them the Turing machine-like structure described above, as $\Z$ represents the indices of the tape.
Abelian-by-cyclic groups have been extensively studied from the point of view of geometry and growth~\cite{farb2000asymptotic, hurtado2021global}, algorithmic problems~\cite{boler1976conjugacy}, random walks~\cite{pittet2003random}, and group algebra isomorphism~\cite{baginski1999isomorphism}.
They also serve as a first step towards understanding general metabelian groups, whose definition is obtained by replacing $\Z$ with an arbitrary abelian group.
Figure~\ref{fig:metabelian} illustrates the relations between the classes of groups introduced above, as well as their known decidability results.

\begin{figure}
    \centering
    \begin{tikzpicture}

        \node [draw, ellipse, thick, black, minimum width=10.2cm, minimum height=5cm, align=center] at   (-5,0)   {};
        \node [align=center, black] at (-5,2.2){\textbf{metabelian}};

        \node [draw, circle, thick, black,fill=blue!20, minimum size=1cm, align=center] at (-8.6,-0.1){$(\Z/p\Z) \wr \Z$};

        \node [draw, circle, thick, black,fill=blue!20, minimum size=1cm, align=center] at (-6.8,-0.8){$\mathsf{BS}(1, p)$};

        \node [draw, circle, thick, black, fill=red!20, minimum size=1cm, align=center] at (-5.2,-0.7){$\Z^m \wr \Z$};

        \node [align=center, black] at (-4.1,-0.5){$\boldsymbol{\cdots}$};

        \node [draw, ellipse, thick, black, fill=red!20, minimum width=2.5cm, minimum height=3.5cm, align=center] at   (-2.1,-0.1)   {};
        \node [align=center, black] at (-2.1,1){\textbf{free} \\ \textbf{metabelian}};

        \node [draw, ellipse, thick, black, pattern=hatch, pattern color=brown!70, hatch size=6pt, minimum width=6.3cm, minimum height=3.5cm, align=center] at   (-6.7,-0.1)   {};
        \node [align=center, black] at (-6.7,1.2){\textbf{abelian-by-cyclic}};

        \matrix [draw,below] at (-5, -3){
          \node [shape=circle, draw=black, fill=red!20, label=right:Known decidability results for Subgroup Intersection~\cite{baumslag2010subgroups}] {}; \\
          \node [shape=circle, draw=black, fill=blue!20, label=right:Known decidability results for Coset Intersection~\cite{DBLP:conf/icalp/CadilhacCZ20,lohrey2015rational}] {}; \\
          \node [shape=circle, draw=black, pattern=hatch, pattern color=brown!70, hatch size=4pt, label=right:Our decidability results for Subgroup Intersection and Coset Intersection] {}; \\
        };
    \end{tikzpicture}
    \caption{Inclusion relation of different classes of metabelian groups.}
    \label{fig:metabelian}
\end{figure}
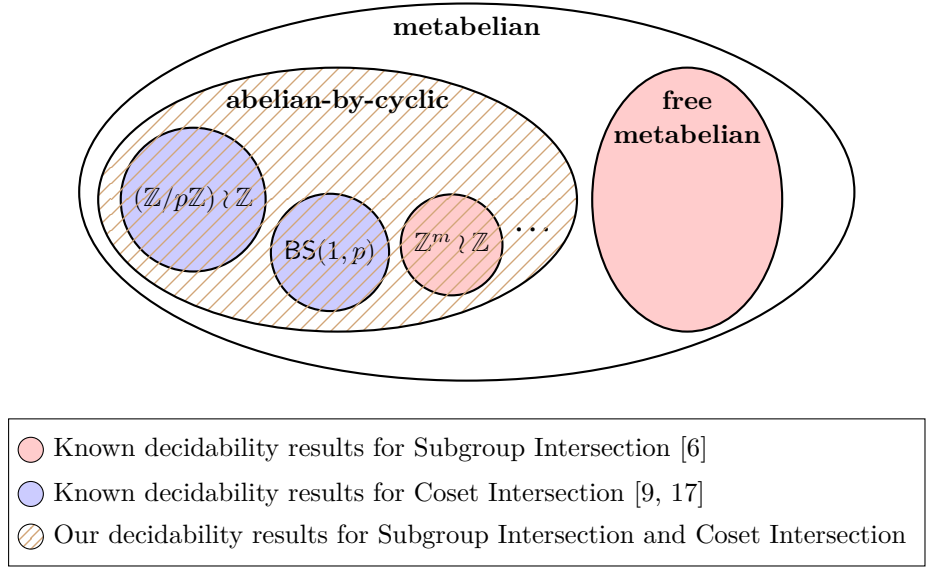

In this paper, we show decidability of Subgroup Intersection and Coset Intersection in finitely generated abelian-by-cyclic groups.
Our approach is different from the automata-based methods~\cite{DBLP:conf/icalp/CadilhacCZ20,lohrey2015rational} used for $(\Z/p\Z) \wr \Z$ and $\mathsf{BS}(1, p)$.
%It is also different from the deep algebraic approach in Baumslag, Miller and Ostheimer's solution to Subgroup Intersection in free metabelian groups~\cite{baumslag2010subgroups}.
%Our approach allows us to solve both Subgroup Intersection and Coset Intersection in abelian-by-cyclic groups.
We reduce both Subgroup and Coset Intersection to the problem of finding an element of the form $X^z - f,\; z \in \Z \setminus \{0\}$ in a given ideal of the Laurent polynomial ring $\Z[X, X^{-1}]$. 
This problem has already been solved by Noskov~\cite{noskov1982conjugacy}.
However, Noskov's solution relies on a series of intricate arguments in commutative algebra.
We propose a more direct solution using a combination of computational algebraic geometry and number theory.

A natural follow-up to our work would be trying to generalize our results to arbitrary metabelian groups.
This boils down to generalizing several arguments in this paper to \emph{multivariate} polynomial rings, which become significantly more difficult.

\section{Preliminaries}

\paragraph*{Laurent polynomial ring and modules}

A (univariate) \emph{Laurent polynomial} with coefficients over $\Z$ is an expression of the form
\[
f = \sum_{i = p}^q a_i X^i, \quad \text{where $p, q \in \Z$ and $a_i \in \Z, i = p, p+1, \ldots, q$.}
\]
The set of all Laurent polynomials with coefficients over $\Z$ forms a ring and is denoted by $\Z[X^{\pm}]$.
On the other hand, we denote by $\Z[X]$ the \emph{usual} univariate polynomial ring over $\Z$: it contains elements whose monomials have non-negative degree.

Let $d \geq 1$ be a positive integer. One can similarly define the Laurent polynomial ring
\[
\Z[X^{\pm d}] \coloneqq \left\{\sum_{i = p}^q a_{di} X^{di} \in \Z[X^{\pm}] \;\middle|\; p, q \in \Z, a_{dp}, \ldots, a_{dq} \in \Z \right\}.
\]
Its elements are Laurent polynomials whose monomials have degrees divisible by $d$.

Let $R$ be a commutative ring.
An $R$-module is defined as an abelian group $(M, +)$ along with an operation $\cdot \;\colon R \times M \rightarrow M$ satisfying $f \cdot (a+b) = f \cdot a + f \cdot b$, $(f + g) \cdot a = f \cdot a + g \cdot a$, $fg \cdot a = f \cdot (g \cdot a)$ and $1 \cdot a = a$.
We will denote by $\bzer$ the neutral element of an $R$-module $M$.

For example, for any $d \in \N$, the group $\Z[X^{\pm}]$ can be seen as a $\Z[X^{\pm d}]$-module by $f \cdot g \coloneqq fg, \; f \in \Z[X^{\pm d}], g \in \Z[X^{\pm}]$.
%The generators of $\Z[X^{\pm}]$ as a $\Z[X^{\pm d}]$-module are $1, X, \ldots, X^{d-1}$.
In general, in order to define a $\Z[X^{\pm d}]$-module structure on an abelian group $M$, it suffices to define $X^d \cdot m$ and $X^{-d} \cdot m$ for all $m \in M$.
The value of $f \cdot m, \; f \in \Z[X^{\pm d}], m \in M,$ would then follow from the linearity of the operation $\cdot$.

An \emph{ideal} of $R$ is a subset of $R$ that is an $R$-module.
If $M$ is an $R$-module and $m \in M$, then $R \cdot m \coloneqq \{r \cdot m \mid r \in R\}$ is again an $R$-module.
If $N$ and $N'$ are $R$-submodule of $M$, then $N + N' \coloneqq \{n + n' \mid n \in N, n' \in N'\}$ is again an $R$-submodule of $M$.

\paragraph*{Finite presentation of modules}

For any $D \in \N$, $\Z[X^{\pm}]^D$ is a $\Z[X^{\pm}]$-module by $f \cdot (g_1, \ldots, g_D) \coloneqq (fg_1, \ldots, fg_D)$.
Throughout this paper, we use the bold symbol $\bff$ to denote a vector $(f_1, \ldots, f_d) \in \Z[X^{\pm}]^D$.
Given $\bg_1, \ldots, \bg_m \in \Z[X^{\pm}]^D$, we say they \emph{generate} the $\Z[X^{\pm}]$-module $\sum_{i=1}^D \Z[X^{\pm}] \cdot \bg_i$.
A module is called \emph{finitely generated} if it can be generated by a finite number of elements.
Given two finitely generated $\Z[X^{\pm}]$-submodules $N, M$ of $\Z[X^{\pm}]^D$ such that $N \subseteq M$, we can define the quotient $M/N \coloneqq \{\overline{\bm} \mid \bm \in M\}$ where $\overline{\bm_1} = \overline{\bm_2}$ if and only if $\bm_1 - \bm_2 \in N$.
This quotient is also an $\Z[X^{\pm}]$-module.
We say that an $\Z[X^{\pm}]$-module $\mA$ is \emph{finitely presented} if it can be written as a quotient $M/N$ for two finitely generated submodules $N \subseteq M$ of $\Z[X^{\pm}]^D$ for some $D \in \N$.
Such a pair $(M, N)$, given by their respective generators, is called a \emph{finite presentation} of $\mA$.
The element $\overline{\bm}$ of $\mA$ is effectively represented by $\bm \in \Z[X^{\pm}]^D$, this representation is unique modulo $N$.

Effective computation in finitely presented modules over polynomial rings is a well-studied area, with numerous algorithms developed to solve a wide range of computation problems.
In particular, these algorithms have been applied to solve various other decision problems in metabelian groups, see the paper~\cite{baumslag1981computable} by Baumslag, Cannonito and Miller for a comprehensive account on this subject.
The following are some classic computational problems with effective algorithms that we will make use of.

\begin{lem}[{\cite[Lemma~2.1,~2.2]{baumslag1981computable}}]\label{lem:classicdec}
    Let $\mA$ be a $\Z[X^{\pm}]$-module with a given finite presentation.
    The following problems are effectively solvable:
    \begin{enumerate}[(i)]
        \item \textit{(Submodule Membership)} Given elements $\ba_1, \ldots, \ba_k, \ba \in \mA$, decide whether $\ba$ is in the submodule generated by $\ba_1, \ldots, \ba_k$.
        \item \textit{(Computing Syzygies)} Given elements $\ba_1, \ldots, \ba_k \in \mA$, compute a finite set of generators for the \emph{Syzygy module} $S \subseteq \Z[X^{\pm}]^k$:
        \[
        S \coloneqq \left\{(f_1, \ldots, f_k) \in \Z[X^{\pm}]^k \;\middle|\; f_1 \cdot \ba_1 + \cdots + f_k \cdot \ba_k = \bzer \right\}.
        \]
        \item \textit{(Computing Intersection)} Given the generators $\ba_1, \ldots, \ba_k$ of a submodule $A \subseteq \mA$ and the generators $\bb_1, \ldots, \bb_m$ of a submodule $B \subseteq \mA$, compute a finite set of generators for the submodule $A \cap B$.
        \setcounter{DecideCounter}{\value{enumi}} 
    \end{enumerate}
    This effectiveness still holds if we replace the Laurent polynomial ring $\Z[X^{\pm}]$ with the regular polynomial ring $\Z[X]$.
\end{lem}
In particular, taking $\mA \coloneqq \Z[X^{\pm}]$, Lemma~\ref{lem:classicdec}(i) becomes the well-known \emph{Ideal Membership} problem: given elements $\ba_1, \ldots, \ba_k, \ba \in \Z[X^{\pm}]$, decide whether $\ba$ is in the ideal generated by $\ba_1, \ldots, \ba_k$.
For Lemma~\ref{lem:classicdec}(ii), it states that one can compute the generators for the solution set of any homogeneous linear equation.
Alternatively, Lemma~\ref{lem:classicdec}(ii) can be understood as a procedure to compute the finite presentation $\Z[X^{\pm}]^k/S$ of the module $\sum_{i = 1}^k \Z[X^{\pm}] \cdot \ba_i$.
%For (iii), we will mainly use it to decide whether the intersection of two submodules is trivial (equal to $\{\bzer\}$).

The following lemma shows we can effectively compute the intersection of a submodule of $\Z[X^{\pm}]^k$ with $\Z^k$.

\begin{lem}[{\cite[Corollary~2.5(2)]{baumslag1981computable}}]\label{lem:decinterZ}
    Suppose we are given $k \in \N$ and elements $\bg_1, \ldots, \bg_n$ of the $\Z[X^{\pm}]$-module $\Z[X^{\pm}]^k$.
    Let $\mM$ denote the $\Z[X^{\pm}]$-module generated by $\bg_1, \ldots, \bg_n$, and define $\Lambda \coloneqq \mM \cap \Z^k$.
    Then $\Lambda \subseteq \Z^k$ is a $\Z$-module, and a finite set of generators for $\Lambda$ can be effectively computed.
\end{lem}

Recall that for any $d \geq 1$, a $\Z[X^{\pm}]$-module is naturally a $\Z[X^{\pm d}]$-module.
In particular, $\Z[X^{\pm}]^D$ is isomorphic as a $\Z[X^{\pm d}]$-module to $\Z[X^{\pm d}]^{Dd}$, and any finitely presented $\Z[X^{\pm}]$-module can be considered as a finitely presented $\Z[X^{\pm d}]$-module:

\begin{restatable}{lem}{lemchangebase}\label{lem:changebase}
    Let $d \geq 2$.
    Given a finite presentation of a $\Z[X^{\pm}]$-module $\mA$, one can compute a finite presentation of $\mA$ as a $\Z[X^{\pm d}]$-module.
    Furthermore, let $\ba \in \mA$ be given in the finite presentation of $\mA$ as $\Z[X^{\pm}]$-module, then one can compute the representation of $\ba$ in $\mA$ considered as a $\Z[X^{\pm d}]$-module.
\end{restatable}

\paragraph*{Abelian-by-cyclic groups}
We now formally define abelian-by-cyclic groups, the main object of study in this paper.
\begin{defn}
    A group $G$ is called \emph{abelian-by-cyclic} if it admits an abelian normal subgroup $A$ such that $G/A \cong \Z$.
\end{defn}

%Let $g \in G$ be any element such that $gA$ generates the infinite cyclic group $G/A$.
%It is a classic result~\cite[Section~3]{baumslag1973subgroups} that the conjugacy action $G/A \times A \rightarrow A, \; (gA, a) \mapsto g^{-1} a g,$ is well-defined (does not depend on the choice of $g$), and gives the abelian subgroup $A$ the structure of a $\Z[X^{\pm}]$-module.
%Here, the action of $X \in \Z[X^{\pm}]$ is defined by the action of $gA$ on $A$; that is, $X \cdot a \coloneqq g^{-1} a g$.
%We will use the calligraphic $\mA$ to denote the abelian group $A$ considered as a $\Z[X^{\pm}]$-module.

It is a classic result~\cite[p.17]{boler1976conjugacy} that every finitely generated abelian-by-cyclic group $G$ can be written as a \emph{semidirect product} $\mA \rtimes \Z$:
\begin{equation}
    \mA \rtimes \Z \coloneqq \left\{ (\ba, z) \;\middle|\; \ba \in \mA, z \in \Z \right\},
\end{equation}
where $\mA$ is a finitely presented $\Z[X^{\pm}]$-module.
The group law in $\mA \rtimes \Z$ is defined by
\[
(\ba, z) \cdot (\ba', z') = (\ba + X^z \cdot \ba', z + z'), \quad (\ba, z)^{-1} = (- X^{-z} \cdot \ba, -z).
\]
The neutral element of $\mA \rtimes \Z$ is $(\bzer, 0)$.
Intuitively, the element $(\ba, z)$ is analogous to a $2 \times 2$ matrix
$
\begin{pmatrix}
X^{z} & \ba \\
0 & 1 \\
\end{pmatrix}
$, where group multiplication is represented by matrix multiplication.
By direct computation, for all $m \in \Z$, we have
\begin{align*}
(\ba, z)^m = 
\begin{cases}
    \left(\frac{X^{mz} - 1}{X^z - 1} \cdot \ba, mz\right), \quad & z \neq 0, \\
    \left(m \cdot \ba, 0\right), \quad & z = 0. \\
\end{cases}
\end{align*}
We naturally identify $\mA$ with the subgroup $\left\{ (\ba, 0) \;\middle|\; \ba \in \mA\right\}$ of $\mA \rtimes \Z$.
In particular, the quotient $\left( \mA \rtimes \Z \right)/\mA$ is isomorphic to $\Z$, so $\mA \rtimes \Z$ is indeed abelian-by-cyclic.

\begin{example}
    If we take $\mA \coloneqq (\Z/p\Z)[X, X^{-1}] = \Z[X^{\pm}]/\left(\Z[X^{\pm}] \cdot p\right)$, then we recover the definition~\eqref{eq:defwr} of the group $(\Z/p\Z) \wr \Z$.
    %Here, $(\Z/p\Z)[X, X^{-1}]$ is the same as the finitely presented module $\Z[X^{\pm}]/\left(\Z[X^{\pm}] \cdot p\right)$.
    If we take $\mA \coloneqq \Z[1/p] = \Z[X^{\pm}]/\left(\Z[X^{\pm}] \cdot (X - p)\right)$, then we recover the definition~\eqref{eq:defbs} of the group $\mathsf{BS}(1, p)$.
    %Here, $\Z[1/p]$ is the same as the finitely presented module $\Z[X^{\pm}]/\left(\Z[X^{\pm}] \cdot (X - p)\right)$.
\end{example}

Throughout this paper, a finitely generated abelian-by-cyclic group $G$ is always represented as the semidirect product $\mA \rtimes \Z$, where $\mA$ is a $\Z[X^{\pm}]$-module given by a finite presentation.

\section{Main results and overview}
The main result of this paper is the following.

\begin{theorem}\label{thm:main}
    Subgroup Intersection and Coset Intersection are decidable for finitely generated abelian-by-cyclic groups.
\end{theorem}

Our proof of Theorem~\ref{thm:main} is divided into two parts.
The first part is to reduce both Subgroup Intersection and Coset Intersection to the \emph{Shifted Monomial Membership} problem:

\begin{defn}\label{def:SMM}
    \emph{Shifted Monomial Membership} is the following decision problem.
    Given as input a finite set of generators of an ideal $\mI \subseteq \Z[X^{\pm}]$, as well as a Laurent polynomial $f \in \Z[X^{\pm}]$, decide if there exists $z \in \Z \setminus \{\bzer\}$ such that $X^z - f \in \mI$.
\end{defn}

The reduction from Subgroup Intersection and Coset Intersection to Shifted Monomial Membership combines various classic techniques from effective computation of finitely presented $\Z[X^{\pm}]$-modules, and will be shown in Section~\ref{sec:reduction}.
The main difficulty in this part is the interaction of modules over different base rings.
Our key is to adaptively change the base rings when combining different modules.

The second part is to prove decidability of Shifted Monomial Membership: this will be shown in Section~\ref{sec:smm}.
As written in the introduction, we provide a more direct proof than that of Noskov~\cite{noskov1982conjugacy}.
We will use structural theorems to classify ideals of $\Z[X]$ and consider each case separately.
In some cases, we employ arguments of \emph{height of algebraic numbers} to produce a bound on $|z|$ whenever $X^z - f \in \mI$.
In other cases, we will show certain periodicity stemming from the finiteness of quotients or from roots of unity.
As a result, in all cases it suffices to verify whether $X^z - f \in \mI$ for a finite number of $z$.

Omitted proofs can be found in Appendix~\ref{app:proof}.
A summary of the algorithms for deciding Subgroup Intersection, Coset Intersection, and Shifted Monomial Membership is given in Appendix~\ref{app:alg}.

\section{Reduction to Shifted Monomial Membership}\label{sec:reduction}
Let $\mA$ be a $\Z[X^{\pm}]$-module given by a finite presentation.
Recall that we naturally identify $\mA$ with the subgroup $\left\{ (\ba, 0) \;\middle|\; \ba \in \mA\right\}$ of $\mA \rtimes \Z$; that is, we will sometimes write $\ba$ instead of $(\ba, 0)$ when the context is clear.
We start with a lemma that effectively describes finitely generated subgroups of $\mA \rtimes \Z$.
Such a description follows from the general description of subgroups of finitely generated metabelian groups~\cite[proof of Theorem~1]{romanovskii1974some}.
Here we give a systematic reformulation in the context of abelian-by-cyclic groups.
\begin{restatable}[Structural theorem of abelian-by-cyclic groups, see also~\cite{romanovskii1974some}]{lem}{lemstruct}\label{lem:struct}
    Let $\gen{\mG}$ be a subgroup of $\mA \rtimes \Z$ generated by the elements $g_1 \coloneqq (\ba_1, z_1), \ldots, g_K \coloneqq (\ba_K, z_K)$.
    Then
    \begin{enumerate}[(i)]
        \item If $z_1 = \cdots = z_K = 0$, then $\gen{\mG}$ is contained in $\mA$ and it is the $\Z$-module generated by $\ba_1, \ldots, \ba_K$.
        \item If $z_1, \ldots, z_K$ are not all zero, then $\gen{\mG} \not\subset \mA$.
        Let $d \in \N$ denote the greatest common divisor of $z_1, \ldots, z_K$. Consider the lattice \[
        \Lambda \coloneqq \left\{(s_1, \ldots, s_K) \in \Z^K \;\middle|\; s_1 z_1 + \cdots + s_K z_K = 0\right\}.
        \]
        Let $(s_{11}, \ldots, s_{1K}), \ldots, (s_{T1}, \ldots, s_{TK})$ be a finite set of generators for $\Lambda$.
        Then $\gen{\mG} \cap \mA$ is a $\Z[X^{\pm d}]$-submodule of $\mA$, generated by the set of elements
        \begin{equation}\label{eq:genmod}
        S \coloneqq \left\{g_i g_j g_i^{-1} g_j^{-1} \;\middle|\; 1 \leq i < j \leq K \right\} \cup \left\{g_1^{s_{i1}} \cdots g_K^{s_{iK}} \;\middle|\; i \in [1, T] \right\}.
        \end{equation}
        \item In case (ii), let $\ba \in \mA$ be any element such that $(\ba, d) \in \gen{\mG}$.
        Then $\gen{\mG}$ is generated by $\gen{\mG} \cap \mA$ and $(\ba, d)$ as a group.
        In other words, every element of $\gen{\mG}$ can be written as $(\bb, 0) \cdot (\ba, d)^m$ for some $\bb \in \gen{\mG} \cap \mA$ and $m \in \Z$.
    \end{enumerate}
\end{restatable}

We point out that in case~(ii), the subgroup $\gen{\mG} \cap \mA$ is finitely generated as a $\Z[X^{\pm d}]$-module; but it is not necessarily finitely generated as a group.

\begin{example}
    Let $\mA = \Z[X^{\pm}]$, considered as a $\Z[X^{\pm}]$-module.
    Let $\gen{\mG}$ be the subgroup of $\mA \rtimes \Z$ generated by the elements $g_1 = (X, 4), g_2 = (1+X, -6)$.
    Then $d = 2$, and $\gen{\mG} \cap \mA$ is the $\Z[X^{\pm 2}]$-module generated by the elements
    \begin{equation*}
    g_1 g_2 g_1^{-1} g_2^{-1} = (X^5 + X^4 - 1 - X^{-5}, 0), \quad
    g_1^3 g_2^2 = (X^{13} + X^{12} + X^{9} + X^{7} + X^{6} + X^{5} + X, 0).
    \end{equation*}

    For example, consider the element $g_1^2 g_2 g_1 g_2 \in \gen{\mG}$.
    By direct computation, its second entry is zero, therefore $g_1^2 g_2 g_1 g_2 \in \gen{\mG} \cap \mA$.
    Furthermore, $g_1^2 g_2 g_1 g_2$ can be written as 
    \begin{multline*}
        g_1^2 g_2 g_1 g_2 = g_1^2 (g_2 g_1) g_2 = g_1^2 (g_2 g_1 g_2^{-1} g_1^{-1}) (g_1 g_2) g_2 = g_1^2 (g_1 g_2 g_1^{-1} g_2^{-1})^{-1} g_1 g_2^2 \\
        = g_1^2 (g_1 g_2^2) (g_1 g_2^2)^{-1} (g_1 g_2 g_1^{-1} g_2^{-1})^{-1} (g_1 g_2^2) = g_1^3 g_2^2 \cdot (g_1 g_2^2)^{-1} (g_1 g_2 g_1^{-1} g_2^{-1})^{-1} (g_1 g_2^2) \\
        = (X^{13} + X^{12} + X^{9} + X^{7} + X^{6} + X^{5} + X, 0) \cdot (g_1 g_2^2)^{-1} (X^5 + X^4 - 1 - X^{-5}, 0)^{-1} (g_1 g_2^2) \\
        = (X^{13} + X^{12} + X^{9} + X^{7} + X^{6} + X^{5} + X) + X^{-8} \cdot (-1) \cdot (X^5 + X^4 - 1 - X^{-5}).
    \end{multline*}
    It is therefore indeed in the $\Z[X^{\pm 2}]$-module generated by $g_1 g_2 g_1^{-1} g_2^{-1} = X^5 + X^4 - 1 - X^{-5}$ and $g_1^3 g_2^2 = X^{13} + X^{12} + X^{9} + X^{7} + X^{6} + X^{5} + X$.

    Intuitively, modulo the generator $g_1 g_2 g_1^{-1} g_2^{-1}$, one can permute letters in any word over $\mG$ (in the above example, $g_1^2 g_2 g_1 g_2$ is congruent to $g_1^3 g_2^2$).
    Whereas the generator $g_1^3 g_2^2$ guarantees the second entry of the product to be zero.
\end{example}

Let $\gen{\mG}, \gen{\mH}$ be finitely generated subgroups of $\mA \rtimes \Z$ given by their respective generators, let $h$ be an element of $\mA \rtimes \Z$.
We now consider Subgroup and Coset Intersection for $\gen{\mG}$ and $\gen{\mH}$.
We split into three cases according to whether $\gen{\mG}$ and $\gen{\mH}$ are contained in the subgroup $\mA$.
If at least one of $\gen{\mG}$ and $\gen{\mH}$ is contained in $\mA$ (Case 1 and 2 below), then the solutions to Subgroup and Coset Intersection are relatively straightforward using the standard tools introduced in Lemma~\ref{lem:classicdec} and \ref{lem:decinterZ}.
They do not need to be reduced to Shifted Monomial Membership.
If neither $\gen{\mG}$ nor $\gen{\mH}$ is contained in the subgroup $\mA$ (Case 3 below), then the solution is more complicated and we reduce both Subgroup and Coset Intersection to Shifted Monomial Membership.

\subsection{Case 1: $\gen{\mG}$ and $\gen{\mH}$ are both contained in $\mA$}

Suppose $\gen{\mG}$ is generated by the elements $g_1 = (\ba_1, 0), \ldots, g_K = (\ba_K, 0)$, and $\gen{\mH}$ is generated by the elements $h_1 = (\ba'_1, 0), \ldots, h_M = (\ba'_M, 0)$.

\paragraph*{Subgroup Intersection}
In this case, we have
\[
\gen{\mG} = \{y_1 \cdot \ba_1 + \cdots + y_K \cdot \ba_K \mid y_1, \ldots, y_K \in \Z\}, \quad \gen{\mH} = \{z_1 \cdot \ba'_1 + \cdots + z_M \cdot \ba'_M \mid z_1, \ldots, z_M \in \Z\}.
\]
Then $\gen{\mG} \cap \gen{\mH} = \{e\}$ if and only if every element of $\gen{\mG} \cap \gen{\mH}$ is equal to the neutral element.
This means that every solution of $y_1 \cdot \ba_1 + \cdots + y_K \cdot \ba_K = z_1 \cdot \ba'_1 + \cdots + z_M \cdot \ba'_M, \; y_1, \ldots, y_K, z_1, \ldots, z_M \in \Z$ is also a solution of $y_1 \cdot \ba_1 + \cdots + y_K \cdot \ba_K = \bzer$.

Let $\mM$ denote the $\Z[X^{\pm}]$-module
\begin{multline}\label{eq:M1s}
\mM \coloneqq \Big\{ (y_1, \ldots, y_K, z_1, \ldots, z_M) \in \Z[X^{\pm}]^{K+M} \;\Big|\; \\
y_1 \cdot \ba_1 + \cdots + y_K \cdot \ba_K - z_1 \cdot \ba'_1 - \cdots - z_M \cdot \ba'_M = \bzer \Big\},
\end{multline}
and let $\mZ$ denote the $\Z[X^{\pm}]$-module
\begin{equation}\label{eq:Z1s}
\mZ \coloneqq \Big\{ (y_1, \ldots, y_K, z_1, \ldots, z_M) \in \Z[X^{\pm}]^{K+M} \;\Big|\; y_1 \cdot \ba_1 + \cdots + y_K \cdot \ba_K = \bzer \Big\}.
\end{equation}
Then the statement above can be summarized as follows.

\begin{obs}\label{obs:C1SI}
    We have $\gen{\mG} \cap \gen{\mH} = \{e\}$ if and only if $\mM \cap \Z^{K+M} = (\mM \cap \mZ) \cap \Z^{K+M}$.
\end{obs}
Indeed, the left hand side $\mM \cap \Z^{K+M}$ denotes all the integer solutions of the equation $y_1 \cdot \ba_1 + \cdots + y_K \cdot \ba_K = z_1 \cdot \ba'_1 + \cdots + z_M \cdot \ba'_M$, while the right hand side $(\mM \cap \mZ) \cap \Z^{K+M}$ denotes all integer solutions of $y_1 \cdot \ba_1 + \cdots + y_K \cdot \ba_K = z_1 \cdot \ba'_1 + \cdots + z_M \cdot \ba'_M = \bzer$.

By Observation~\ref{obs:C1SI}, we can decide Subgroup Intersection in this case.
By Lemma~\ref{lem:classicdec}(ii) we can compute the generators of $\mM$ and $\mZ$, then by Lemma~\ref{lem:classicdec}(iii) we can compute the generators of $\mM \cap \mZ$.
Next, by Lemma~\ref{lem:decinterZ} we can compute the generators of $\mM \cap \Z^{K+M}$ and $(\mM \cap \mZ) \cap \Z^{K+M}$.
Since these are subgroups of $\Z^{K+M}$, their equality can be decided by checking whether all generators of one subgroup belong to the other subgroup.

\paragraph*{Coset Intersection}
Let $h = (\ba_h, z_h)$.
If $z_h \neq 0$ then $\gen{\mG} \cap h \gen{\mH} = \emptyset$.
Therefore we only need to consider the case where $z_h = 0$.
Then $\gen{\mG} \cap h \gen{\mH} = \emptyset$ if and only if there is no solution for $y_1 \cdot \ba_1 + \cdots + y_K \cdot \ba_K = z_1 \cdot \ba'_1 + \cdots + z_M \cdot \ba'_M + z \cdot \ba_h, \; y_1, \ldots, y_K, z_1, \ldots, z_M \in \Z, z = 1$.

Let $\mM'$ denote the $\Z[X^{\pm}]$-module
\begin{multline}\label{eq:M1c}
\mM' \coloneqq \Big\{ (y_1, \ldots, y_K, z_1, \ldots, z_M, z) \in \Z[X^{\pm}]^{K+M+1} \;\Big|\; \\
y_1 \cdot \ba_1 + \cdots + y_K \cdot \ba_K - z_1 \cdot \ba'_1 - \cdots - z_M \cdot \ba'_M - z \cdot \ba_h = \bzer \Big\}.
\end{multline}

\begin{obs}
    We have $\gen{\mG} \cap h \gen{\mH} = \emptyset$ if and only if $\big(\mM' \cap \Z^{K+M+1}\big) \cap \big(\Z^{K+M} \times \{1\}\big) = \emptyset$.
\end{obs}

Again, $\mM' \cap \Z^{K+M+1}$ can be computed by Lemma~\ref{lem:classicdec}(ii)(iii). So Coset Intersection in this case can be decided using linear algebra over $\Z$.

\subsection{Case 2: one of $\gen{\mG}$ and $\gen{\mH}$ is contained in $\mA$}

This case is similar to the previous case, we leave the detailed proofs in the appendix and summarize the result by the following proposition.

\begin{restatable}{prop}{propCtwo}\label{prop:C2}
    Suppose exactly one of $\gen{\mG}$ and $\gen{\mH}$ is contained in $\mA$.
    Let $h \in \mA \rtimes \Z$.
    Given finite sets of generators of $\gen{\mG}$ and $\gen{\mH}$ as groups, it is decidable whether $\gen{\mG} \cap \gen{\mH} = \{e\}$ and whether $\gen{\mG} \cap h \gen{\mH} = \emptyset$.
\end{restatable}

\subsection{Case 3: neither $\gen{\mG}$ nor $\gen{\mH}$ is contained in $\mA$}

By Lemma~\ref{lem:struct}, suppose $\gen{\mG}$ is generated by $\gen{\mG} \cap \mA$ and an element $(\ba_{\mG}, d_{\mG})$; and $\gen{\mH}$ is generated by $\gen{\mH} \cap \mA$ and an element $(\ba_{\mH}, d_{\mH})$.
The elements $(\ba_{\mG}, d_{\mG})$ and $(\ba_{\mH}, d_{\mH})$ can be effectively computed from the generating sets $\mG, \mH$ by performing the Euclidean algorithm.
Furthermore, $\gen{\mG} \cap \mA$ is a $\Z[X^{\pm d_{\mG}}]$-module whose generators are explicitly given (by Equation~\eqref{eq:genmod}), and $\gen{\mH} \cap \mA$ is a $\Z[X^{\pm d_{\mH}}]$-module whose generators are explicitly given.

\paragraph*{Subgroup Intersection}
We have $\gen{\mG} \cap \gen{\mH} = \{e\}$ if and only if the equation
$
    (\bb, 0) \cdot (\ba_{\mG}, d_{\mG})^m = (\bc, 0) \cdot (\ba_{\mH}, d_{\mH})^n
$
has non-trivial solutions $\bb \in \gen{\mG} \cap \mA,\; \bc \in \gen{\mH} \cap \mA,\; m, n \in \Z$.
Here, non-trivial means $\bb, \bc, m, n$ are not all zero.
By direct computation, this equation is equivalent to the system
\begin{equation}\label{eq:intersys}
    \bb + \frac{X^{m d_{\mG}} - 1}{X^{d_{\mG}} - 1} \cdot \ba_{\mG} = \bc + \frac{X^{n d_{\mH}} - 1}{X^{d_{\mH}} - 1} \cdot \ba_{\mH}, \quad m d_{\mG} = n d_{\mH}.
\end{equation}
An obstacle here is that $\bb$ and $\bc$ take values respectively in the $\Z[X^{\pm d_{\mG}}]$-module $\gen{\mG} \cap \mA$ and the $\Z[X^{\pm d_{\mH}}]$-module $\gen{\mH} \cap \mA$.
This difference makes solving \eqref{eq:intersys} complicated.
To overcome this obstacle we define $d \coloneqq \lcm(d_{\mG}, d_{\mH})$; that is, $d$ is the smallest positive integer such that $d_{\mG} \mid d, \; d_{\mH} \mid d$.
We can thus consider both $\gen{\mG} \cap \mA$ and $\gen{\mH} \cap \mA$ as $\Z[X^{\pm d}]$-modules (see Lemma~\ref{lem:changebase}).
%Indeed, since $\gen{\mG} \cap \mA$ is a $\Z[X^{\pm d_{\mG}}]$-module and $d_{\mG} \mid d$, it can be naturally considered as a $\Z[X^{\pm d}]$-module.
%Similarly, the $\Z[X^{\pm d_{\mH}}]$-module $\gen{\mH} \cap \mA$ is naturally a $\Z[X^{\pm d}]$-module.
One can compute a finite set of generators $S_{\mG}$ for $\gen{\mG} \cap \mA$ as a $\Z[X^{\pm d}]$-module.
Similarly, one can compute a finite set of generators $S_{\mH}$ for $\gen{\mH} \cap \mA$ as a $\Z[X^{\pm d}]$-module.

Define the $\Z[X^{\pm d}]$-module 
$
    \mM \coloneqq (\gen{\mG} \cap \mA) + (\gen{\mH} \cap \mA)
$,
it is generated by the set $S_{\mG} \cup S_{\mH}$.

\begin{restatable}{lem}{lemintertoeq}\label{lem:intertoeq}
    The intersection $\gen{\mG} \cap \gen{\mH}$ is non-trivial if and only if at least one of the following two conditions is satisfied:
    \begin{enumerate}[(i)]
        \item $(\gen{\mG} \cap \mA) \cap (\gen{\mH} \cap \mA) \neq \{\bzer\}$.
        \item The equation
            \begin{equation}
                (X^{zd} - 1) \cdot \ba_{\mG, \mH} \in (X^d - 1) \cdot \mM
            \end{equation}
            has solution $z \in \Z \setminus \{\bzer\}$.
            Here,
            \[
            \ba_{\mG, \mH} \coloneqq \frac{X^d - 1}{X^{d_{\mG}} - 1} \cdot \ba_{\mG} - \frac{X^d - 1}{X^{d_{\mH}} - 1} \cdot \ba_{\mH}.
            \]
    \end{enumerate}
\end{restatable}
\begin{proof}
    Suppose $\gen{\mG} \cap \gen{\mH}$ is non-trivial.
    Let $(\ba, z') \in \gen{\mG} \cap \gen{\mH}$, then $d_{\mG} \mid z'$ and $d_{\mH} \mid z'$.
    Therefore $d \mid z'$ and we can write $z' = zd$ for some $z \in \Z$.

    If $z = 0$ then $\ba \in (\gen{\mG} \cap \mA) \cap (\gen{\mH} \cap \mA)$, so $(\gen{\mG} \cap \mA) \cap (\gen{\mH} \cap \mA) \neq \{\bzer\}$.
    If $z \neq 0$ then Equation~\eqref{eq:intersys} has solution with $m d_{\mG} = n d_{\mH} = z d$, meaning
    \[
    \frac{X^{zd} - 1}{X^d - 1} \cdot \left(\frac{X^d - 1}{X^{d_{\mG}} - 1} \cdot \ba_{\mG} - \frac{X^d - 1}{X^{d_{\mH}} - 1} \cdot \ba_{\mH}\right) = \bc - \bb.
    \]
    In particular, we have $\bc - \bb \in (\gen{\mG} \cap \mA) + (\gen{\mH} \cap \mA) = \mM$.
    Multiplying both sides by $X^d - 1$ yields $(X^{zd} - 1) \cdot \ba_{\mG, \mH} \in (X^d - 1) \cdot \mM$.

    Suppose either (i) or (ii) is satisfied. In case~(i), we have $\gen{\mG} \cap \gen{\mH} \supseteq (\gen{\mG} \cap \mA) \cap (\gen{\mH} \cap \mA) \neq \{\bzer\}$.
    In case~(ii), we have
    \[
    \frac{X^{zd} - 1}{X^d - 1} \cdot \ba_{\mG, \mH} \in \mM = (\gen{\mG} \cap \mA) + (\gen{\mH} \cap \mA),
    \]
    so it can be written as $\bc - \bb$ for some $\bc \in \gen{\mH} \cap \mA, \; \bb \in \gen{\mG} \cap \mA$.
    Therefore Equation~\eqref{eq:intersys} has non-trivial solutions by taking $m \coloneqq zd/d_{\mG}, \, n \coloneqq zd/d_{\mH}$.
\end{proof}

\begin{prop}\label{prop:interredsmm}
    Suppose neither $\gen{\mG}$ nor $\gen{\mH}$ is contained in $\mA$. Then deciding whether $\gen{\mG} \cap \gen{\mH} = \{e\}$ reduces to Shifted Monomial Membership (see Definition~\ref{def:SMM}).
\end{prop}
\begin{proof}
    (See Algorithm~\ref{alg:groupinter} for a summary.)
    We use Lemma~\ref{lem:intertoeq}. Condition~(i) can be decided using Lemma~\ref{lem:classicdec}(iii), so it suffices to decide condition~(ii).
    The set
    \begin{equation}\label{eq:defI}
    \mI \coloneqq \left\{f \in \Z[X^{\pm d}] \;\middle|\; f \cdot \ba_{\mG, \mH} \in (X^d - 1) \cdot \mM \right\}
    \end{equation}
    is an ideal of $\Z[X^{\pm d}]$.
    Furthermore, given a finite set of generators for $\mM$, one can compute a finite set of generators for $\mI$ using Lemma~\ref{lem:classicdec}(ii).
    
    Condition~(ii) in Lemma~\ref{lem:intertoeq} is equivalent to ``$\mI$ contains an element $X^{zd} - 1$ for some $z \in \Z \setminus \{0\}$''.
    Performing the variable change $X^d \rightarrow X$, we can consider $\mI$ as an ideal of $\Z[X^{\pm}]$ and the condition becomes ``$\mI$ contains an element $X^{z} - 1$ for some $z \in \Z \setminus \{0\}$''.
    This is exactly the Shifted Monomial Membership problem (Definition~\ref{def:SMM}) with $f = 1$.
\end{proof}

\begin{rmk}\label{rmk:red}
    The fact that $\mM \coloneqq (\gen{\mG} \cap \mA) + (\gen{\mH} \cap \mA)$ is a finitely generated $\Z[X^{\pm d}]$-module is crucial to the reduction in Proposition~\ref{prop:interredsmm}.
    This argument is specific to abelian-by-cyclic groups and no longer holds in arbitrary metabelian groups.
    
    For example, let $\mA$ now be a finitely presented module over the bivariate Laurent polynomial ring $\Z[X^{\pm}, Y^{\pm}]$.
    We can similarly define the semidirect product $\mA \rtimes \Z^2$, which is metabelian but not abelian-by-cyclic.
    We can find subgroups $\gen{\mG}$ and $\gen{\mH}$ such that $\gen{\mG} \cap \mA$ is a finitely generated $\Z[X^{\pm}]$-module, while $\gen{\mH} \cap \mA$ is a finitely generated $\Z[Y^{\pm}]$-module.
    In this case, if we define the sum $\mM \coloneqq (\gen{\mG} \cap \mA) + (\gen{\mH} \cap \mA)$, then $\mM$ is not a $\Z[X^{\pm d}]$-module, a $\Z[Y^{\pm d}]$-module, or a $\Z[X^{\pm d}, Y^{\pm d}]$-module for any $d \geq 1$.
    While $\mM$ is still a $\Z$-module (since both $\Z[X^{\pm}]$-modules and $\Z[Y^{\pm}]$-modules can be seen as $\Z$-modules), it is not finitely generated as a $\Z$-module.

    Since being finitely generated is essential to the effectiveness results in Lemma~\ref{lem:classicdec}-\ref{lem:changebase}, this constitutes the key difficulty in generalizing our results from abelian-by-cyclic groups to arbitrary metabelian groups.
    The same difficulty also appears in Coset Intersection.
\end{rmk}

\paragraph*{Coset Intersection}
Let $h = (\ba_h, z_h)$.
Then $\gen{\mG} \cap h \gen{\mH} = \emptyset$ if and only if the equation
$
    (\bb, 0) \cdot (\ba_{\mG}, d_{\mG})^m = (\ba_h, z_h) \cdot (\bc, 0) \cdot (\ba_{\mH}, d_{\mH})^n
$
has solutions $\bb \in \gen{\mG} \cap \mA, \bc \in \gen{\mH} \cap \mA, m, n \in \Z$.
By direct computation, this is equivalent to the system
\begin{equation}\label{eq:Cosetsys}
    \bb + \frac{X^{m d_{\mG}} - 1}{X^{d_{\mG}} - 1} \cdot \ba_{\mG} = \ba_h + X^{z_h} \cdot \bc + X^{z_h} \cdot \frac{X^{n d_{\mH}} - 1}{X^{d_{\mH}} - 1} \cdot \ba_{\mH}, \quad m d_{\mG} = n d_{\mH} + z_h.
\end{equation}

Again we define $d \coloneqq \lcm(d_{\mG}, d_{\mH})$ and consider both $\gen{\mG} \cap \mA$ and $\gen{\mH} \cap \mA$ as $\Z[X^{\pm d}]$-modules, respectively generated by the sets $S_{\mG}$ and $S_{\mH}$.
This time, we define the $\Z[X^{\pm d}]$-module 
$
    \mM' \coloneqq (\gen{\mG} \cap \mA) + X^{z_h} \cdot (\gen{\mH} \cap \mA)
$,
it is generated by the set $S_{\mG} \cup (X^{z_h} \cdot S_{\mH})$.

If $m d_{\mG} = n d_{\mH} + z_h$ has no integer solutions $m, n$, then $\gen{\mG} \cap h \gen{\mH} = \emptyset$.
Otherwise, there exist $z_{\mG} \coloneqq m d_{\mG}, z_{\mH} \coloneqq n d_{\mH} \in \Z$ such that $d_{\mG} \mid z_{\mG},\; d_{\mH} \mid z_{\mH}$ and $z_{\mG} = z_{\mH} + z_h$.
Then, every solution $(m, n) \in \Z^2$ of the equation $m d_{\mG} = n d_{\mH} + z_h$ is of the form
\[
m = (z_{\mG} + z d)/d_{\mG}, \quad n = (z_{\mH} + z d)/d_{\mH}, \quad z \in \Z.
\]
Similar to Lemma~\ref{lem:intertoeq}, we can show the following:

\begin{restatable}{lem}{lemcosettoeq}\label{lem:cosettoeq}
    Let $z_{\mG}, z_{\mH}$ be integers such that $d_{\mG} \mid z_{\mG},\; d_{\mH} \mid z_{\mH}$ and $z_{\mG} = z_{\mH} + z_h$.
    The intersection $\gen{\mG} \cap h \gen{\mH}$ is non-empty if and only if the equation
        \begin{equation}\label{eq:Coseteqz}
            X^{z d} \cdot \ba'_{\mG, \mH} - \ba''_{\mG, \mH} \in (X^d - 1) \cdot \mM'
        \end{equation}
        has solution $z \in \Z$.
        Here,
        \[
        \ba'_{\mG, \mH} \coloneqq X^{z_{\mG}} \cdot \frac{X^d - 1}{X^{d_{\mG}} - 1} \cdot \ba_{\mG} - X^{z_{\mH}} \cdot \frac{X^d - 1}{X^{d_{\mH}} - 1} \cdot \ba_{\mH},
        \]
        and
        \[
        \ba''_{\mG, \mH} \coloneqq \frac{X^d - 1}{X^{d_{\mG}} - 1} \cdot \ba_{\mG} - \frac{X^d - 1}{X^{d_{\mH}} - 1} \cdot \ba_{\mH} + (X^d - 1) \cdot \ba_h.
        \]
\end{restatable}

We can decide if 
\begin{equation}\label{eq:faainM}
f \cdot \ba'_{\mG, \mH} - \ba''_{\mG, \mH} \in (X^d - 1) \cdot \mM'
\end{equation}
has solution $f \in \Z[X^{\pm d}]$ by deciding membership of $\ba''_{\mG, \mH}$ in the $\Z[X^{\pm d}]$-module generated by $\ba'_{\mG, \mH}$ and $\mM'$ (see Lemma~\ref{lem:classicdec}(i)).
If Equation~\eqref{eq:faainM} does not have a solution $f \in \Z[X^{\pm d}]$, then \eqref{eq:Coseteqz} cannot have a solution $z \in \Z$; otherwise, we can compute a solution $f = f_0$ of Equation~\eqref{eq:faainM}.
For example, $f_0$ can be computed by enumerating all elements $f \in \Z[X^{\pm d}]$ (there are countably many), and test for each one whether it satisfies Equation~\eqref{eq:faainM}.
Since Equation~\eqref{eq:faainM} has a solution, this procedure must terminate.

Consider the ideal
\begin{equation}\label{eq:defIp}
\mI' \coloneqq \left\{f \in \Z[X^{\pm d}] \;\middle|\; f \cdot \ba'_{\mG, \mH} \in (X^d - 1) \cdot \mM' \right\}
\end{equation}
of $\Z[X^{\pm d}]$, then a finite set of generators for $\mI'$ can be computed by Lemma~\ref{lem:classicdec}(ii).

\begin{restatable}{lem}{lemctoi}\label{lem:Cosettoideal}
    The solution set
    \[
    \left\{f \in \Z[X^{\pm d}] \;\middle|\; f \cdot \ba'_{\mG, \mH} - \ba''_{\mG, \mH} \in (X^d - 1) \cdot \mM' \right\}
    \]
    is equal to $f_0 + \mI' \coloneqq \{f_0 + g \mid g \in \mI'\}$.
\end{restatable}

\begin{prop}
    Suppose neither $\gen{\mG}$ nor $\gen{\mH}$ is contained in $\mA$.
    Then deciding whether $\gen{\mG} \cap h \gen{\mH} = \emptyset$ reduces to Shifted Monomial Membership.
\end{prop}
\begin{proof}
    (See Algorithm~\ref{alg:cosetinter} for a summary.)
    Suppose there exist $z_{\mG}, z_{\mH} \in \Z$ such that $d_{\mG} \mid z_{\mG},\; d_{\mH} \mid z_{\mH}$ and $z_{\mG} = z_{\mH} + z_h$, otherwise $\gen{\mG} \cap h \gen{\mH} = \emptyset$.
    By Lemma~\ref{lem:cosettoeq} and \ref{lem:Cosettoideal}, it suffices to decide whether there exists $z \in \Z$ such that $X^z \in f_0 + \mI'$.
    This is equivalent to $X^z - f_0 \in \mI'$. 
    We can decide whether $X^0 - f_0 \in \mI'$ using ideal membership of $1 - f_0$ in $\mI'$. Then we use Shifted Monomial Membership to decide whether there exists $z \in \Z \setminus \{\bzer\}$ such that $X^z - f_0 \in \mI'$.
\end{proof}

\section{Deciding Shifted Monomial Membership}\label{sec:smm}

In this section we show that Shifted Monomial Membership is decidable.
Recall that for this problem, we are given a finite set of generators of an ideal $\mI \subseteq \Z[X^{\pm}]$, as well as a Laurent polynomial $f \in \Z[X^{\pm}]$.
We want to decide if there exists $z \in \Z \setminus \{\bzer\}$ such that $X^z - f \in \mI$.

The outline of the proof is as follows.
In Subsection~\ref{subsec:redzx} we first simplify the problem by reducing to ideals $\tI$ over the ring $\Z[X]$ instead of $\mI \subseteq \Z[X^{\pm}]$.
We then consider the greatest common divisor $\varphi$ of the elements in $\tI$, and divide into five cases according to $\varphi$.
Each of Subsections~\ref{subsec:caseone}-\ref{subsec:casefive} treats a separate case.
A common idea in each case is to give a bound on the absolute value of $z$ whenever Shifted Monomial Membership has positive answer.
%This is achieved either by proving certain periodicity, or by producing a concrete bound using number theory tools such as \emph{height} of algebraic numbers.
See Algorithm~\ref{alg:smm} for a summary.

\subsection{Reduction to ideals of $\Z[X]$}\label{subsec:redzx}

Let $g_1, \ldots, g_m$ be the given generators of the ideal $\mI \subseteq \Z[X^{\pm}]$.
Without loss of generality suppose none of the $g_i$ is zero.
Multiplying any $g_i$ with any power of $X$ does not change the ideal they generate, because $X$ is invertible in $\Z[X^{\pm}]$.
Therefore we can multiply each $g_i$ with a suitable power of $X$, and without loss of generality suppose $g_1, \ldots, g_m$ are polynomials in $\Z[X]$ instead of $\Z[X^{\pm}]$, and that they are not divisible by $X$.
Let $\tI$ denote the ideal of $\Z[X]$ generated by $g_1, \ldots, g_m$.

\begin{restatable}{lem}{lemlaurenttoreg}\label{lem:laurenttoreg}
    Let $g$ be a polynomial in $\Z[X^{\pm}]$.
    Then $g \in \mI$ if and only if for some $c \in \N$, $X^c \cdot g \in \tI$.
\end{restatable}

By Lemma~\ref{lem:laurenttoreg}, the ideal $\mI \subseteq \Z[X^{\pm}]$ contains an element $X^z - f$ for some $z \neq 0$, if and only if $\tI$ contains an element $X^a - X^b f$ for some $a, b \in \N, a \neq b$.
Furthermore, in this case, we have $z = a - b$.

Hence, Shifted Monomial Membership reduces to the following problem:
\begin{problem}
    Given the generators of an ideal $\tI \subseteq \Z[X]$, decide whether $\tI$ contains any element of the form $X^a - X^b f$, $a, b \in \N, a \neq b$.
\end{problem}

%For a set of elements $S$ in the ring $\Z[X]$, one can define the notion of their \emph{greatest common divisor} similar to the ring $\Z$.
For a non-zero polynomial $g \in \Z[X]$, its \emph{leading coefficient} is defined as the coefficient of its monomial of largest degree.
For example, the leading coefficient of $3X^2 + 4$ is $3$.
A \emph{common divisor} of a set $S \subseteq \Z[X]$ is a polynomial $g$ with positive leading coefficient, such that $g \mid s$ for all $s \in S$. The \emph{greatest common divisor} of $S$, denoted by $\gcd(S)$, is a polynomial that has the largest degree and largest leading coefficient among all common divisors of $S$.
The greatest common divisor is well-defined over $\Z[X]$ because it is a Unique Factorization Domain~\cite{sharpe1987rings}.
In particular, as $\tI \subseteq \Z[X]$ is the ideal generated by $g_1, \ldots, g_m$, the greatest common divisor $\gcd(\tI)$ is equal to $\gcd(\{g_1, \ldots, g_m\})$.

Denote
$
\varphi \coloneqq \gcd(\tI)
$.
Then $X \nmid \varphi$ because $X \nmid g_1$.
We say that a polynomial $g \in \Z[X]$ is \emph{primitive} if there is no integer $d \geq 2$ such that $d \mid g$.
A complex number $x$ is called a \emph{root of unity} if $x^p = 1$ for some $p \geq 1$.
We say that a polynomial $g \in \Z[X]$ has a \emph{square divisor} if $\phi^2 \mid g$ for some polynomial $\phi \in \Z[X]$ with degree at most one.
A polynomial is called \emph{square-free} if it does not have a square divisor.
Since $\varphi \neq 0$, there are only five cases regarding $\varphi$:
\begin{enumerate}[(i)]
    \item $\varphi = 1$,
    \item $\varphi$ is not primitive,
    \item $\varphi$ is primitive and has a root that is not a root of unity,
    \item $\varphi$ is primitive, all roots of $\varphi$ are roots of unity, and $\varphi$ has a square divisor,
    \item $\varphi$ is primitive, all roots of $\varphi$ are roots of unity, and $\varphi$ is square-free.
\end{enumerate}
Each of the following subsections deals with one case.

\subsection{Case (i): trivial GCD}\label{subsec:caseone}
In this case, $\varphi = 1$. The following lemma gives the structure of the ideal $\tI$ in this case.
A polynomial in $\Z[X]$ is called \emph{monic} if its leading coefficient is one.

\begin{lem}[{\cite[p.384-385]{szekeres1952canonical}}]\label{lem:idealstr}
    Let $\tI$ be an ideal of $\Z[X]$ such that $\gcd(\tI) = 1$.
    Then there are only two possible cases for $\tI$:
    \begin{enumerate}[(i)]
        \item either $\tI = \Z[X]$,
        \item or $\tI$ contains an integer $c \geq 2$, as well as a monic polynomial $g$ of degree at least one.
    \end{enumerate}
    Furthermore, given a finite set of generators for $\tI$, one can decide which case is true. In case~(ii), one can explicitly compute such $c$ and $g$.
\end{lem}

If $\tI = \Z[X]$ then obviously it contains an element $X^a - X^b f, \; a \neq b$.
Suppose now that $\tI$ contains an integer $c \geq 2$ and a monic polynomial $g$ of degree at least one.
In particular, $\tI \subseteq (\Z[X] \cdot g + \Z[X] \cdot c)$.

\begin{lem}\label{lem:finring}
    The quotient $\Z[X] / (\Z[X] \cdot g + \Z[X] \cdot c)$ is finite.
\end{lem}
\begin{proof}
    Let $\deg g$ denote the degree of $g$.
    Since $g$ is monic, every $f \in \Z[X]$ can be written as $f = gh + r$ where $g, h, r \in \Z[X]$ and $\deg r < \deg g$.
    Therefore, every element in $\Z[X]$ is equivalent modulo $g$ to a polynomial with degree at most $\deg g - 1$.
    But there are only finitely many polynomials modulo $c$ with degree at most $\deg g - 1$.
    Therefore, the quotient $\Z[X] / (\Z[X] \cdot g + \Z[X] \cdot c)$ is finite.
\end{proof}

Let $f \mapsto \overline{f}$ denote the canonical projection $\Z[X] \rightarrow\Z[X] / (\Z[X] \cdot g + \Z[X] \cdot c)$.
Consider the sequence $\overline{1}, \overline{X}, \overline{X^2}, \cdots \in \Z[X] / (\Z[X] \cdot g + \Z[X] \cdot c)$.
Since $\Z[X] / (\Z[X] \cdot g + \Z[X] \cdot c)$ is finite, there exists $0 \leq p < q$ such that $\overline{X^p} = \overline{X^q}$.
Furthermore, such integers $p, q$ can be effectively found by incrementally testing whether $X^q - X^p \in (\Z[X] \cdot g + \Z[X] \cdot c)$ (see Lemma~\ref{lem:classicdec}).
Then $X^q - X^p \in (\Z[X] \cdot g + \Z[X] \cdot c)$, so $\overline{X^r} = \overline{X^{r - (q-p)}}$ for every $r \geq q$.
From this, we easily obtain the following result.

\begin{restatable}{lem}{lemgcdone}\label{lem:gcdone}
    Suppose $\tI$ contains an integer $c \geq 2$, as well as a monic polynomial $g$ of degree at least one.
    Then $\tI$ contains an element of the form $X^a - X^b f$, $a, b \in \N, a \neq b$, if and only if $\tI$ contains an element of the form $X^{a'} - X^{b'} f$, $a', b' \in [0, q-1]$.
\end{restatable}
\begin{proof}
    Since $\overline{X^r} = \overline{X^{r - (q-p)}}$ for every $r \geq q$, every $X^r, \; r \in \N$ is equivalent modulo $(\Z[X] \cdot g + \Z[X] \cdot c)$ to $X^{r'}$ for some $r' \in [0, q-1]$.
    Since $(\Z[X] \cdot g + \Z[X] \cdot c) \subseteq \tI$, every $X^r$ is also equivalent modulo $\tI$ to $X^{r'}$ for some $r' \in [0, q-1]$.
    Therefore, if $\tI$ contains an element of the form $X^a - X^b f$, $a, b \in \N, a \neq b$, then $\tI$ contains an element of the form $X^{a'} - X^{b'} f$, $a', b' \in [0, q-1]$.
    And if $\tI$ contains some $X^{a'} - X^{b'} f$, $a', b' \in [0, q-1]$, then it also contains $X^{a' + q(q - p)} - X^{b'} f$.
    Taking $a \coloneqq a' + q(q - p), b \coloneqq b'$ we have $X^a - X^b f \in \tI$ and $a \neq b$.
\end{proof}

Since there are only finitely many integers in $[0, q-1]$, one can decide whether $\tI$ contains an element of the form $X^{a'} - X^{b'} f$, $a', b' \in [0, q-1]$ by enumerating all such $a', b'$.

\subsection{Case (ii): non-primitive GCD}

In this case, $\varphi$ is not primitive.
Suppose $d \mid \varphi$ with $d \geq 2$.
Then $d$ divides every element in $\tI$.
We show that there is an effectively computable bound on $a - b$.

\begin{lem}
    Let $d \geq 2$.
    If $d \mid X^a - X^b f$, then $0 \leq a - b \leq \deg f$.
\end{lem}
\begin{proof}
    If $a > b + \deg f$ then $\deg X^a > \deg X^b f$, so the leading coefficient of $X^a - X^b f$ is $1$, a contradiction to $d \mid X^a - X^b f$.
    Similarly if $a < b$ then the coefficient of the monomial $X^a$ in $X^a - X^b f$ is one, a contradiction to $d \mid X^a - X^b f$.
    Therefore $0 \leq a - b \leq \deg f$.
\end{proof}

Therefore if $X^a - X^b f \in \tI, \; a \neq b,$ then $d \mid X^a - X^b f$, and we must have $a - b \in [1, \deg f]$.
By Lemma~\ref{lem:laurenttoreg}, we have $X^a - X^b f \in \tI$ if and only if $X^{a - b} - f \in \mI$.
Therefore in this case, it suffices to decide for each $r \in [1, \deg f]$ whether $X^r - f \in \mI$.

\subsection{Case (iii): non-root of unity}

In this case, $\varphi$ has a root $x$ that is not a root of unity.
Since $X \nmid g_1$ we have $X \nmid \varphi$, so $x \neq 0$.
Let $\K$ be an algebraic number field that contains $x$.
The key idea in this case is to use the \emph{height function} over $\K$ to give a bound on $|a - b|$.
For an exact construction of the height function, see~\cite[Section~3.2]{waldschmidt2013diophantine}.
In this paper we will only make use of its properties listed in the following lemma.

\begin{lem}[{Height of algebraic numbers~\cite[Property~3.3 and Section~3.6]{waldschmidt2013diophantine}}]\label{lem:height}
    Let $\K$ be an algebraic number field and denote $\K^* \coloneqq \K \setminus \{0\}$. 
    There exists a map $H \colon \K^* \rightarrow \Rp$ that satisfies to following properties.
    \begin{enumerate}[(i)]
        \item For any $n \in \Z$ and $y \in \K^*$, we have $H(y^n) = H(y)^{|n|}$.
        \item For all $y \in \K^*$, we have $H(y) \geq 1$. And $H(y) = 1$ if and only if $y$ is a root of unity.
    \end{enumerate}
    For any $y \in \K^*$, the value $H(y)$ is called the \emph{height} of $y$, it is an algebraic number that can be effectively computed.
\end{lem}

Since $x$ is not a root of unity, we have $H(x) > 1$.

\begin{lem}\label{lem:ht}
    Let $x \neq 0$ be a root of $\varphi$ that is not a root of unity. 
    If $\varphi \mid X^a - X^b f$, then $f(x) \neq 0$ and $|a - b| = \frac{\log H(f(x))}{\log H(x)}$.
\end{lem}
\begin{proof}
    Since $\varphi \mid X^a - X^b f$ and $x$ is a root of $\varphi$, we have $x^a - x^b f(x) = 0$.
    Therefore $x^{a - b} = f(x)$. Since $x \neq 0$ we have $f(x) \neq 0$.
    Taking the height function on both sides of $x^{a - b} = f(x)$ yields $H(x)^{|a - b|} = H(x^{a - b}) = H(f(x))$, so $|a - b| = \frac{\log H(f(x))}{\log H(x)}$.
\end{proof}

If $X^a - X^b f \in \tI$ then we must have $\varphi \mid X^a - X^b f$.
%Since $a \neq b$ we have $a - b \in \left[- \frac{\log H(f(x))}{\log H(x)}, -1\right] \cup \left[1, \frac{\log H(f(x))}{\log H(x)} \right]$.
By Lemma~\ref{lem:laurenttoreg}, we have $X^a - X^b f \in \tI$ if and only if $X^{a - b} - f \in \mI$.
%Therefore similar to the previous case, it suffices to decide for each integer $r \in \left[- \frac{\log H(f(x))}{\log H(x)}, -1\right] \cup \left[1, \frac{\log H(f(x))}{\log H(x)} \right]$ whether $X^r - f \in \mI$.
Therefore by Lemma~\ref{lem:ht}, it suffices to decide whether $r \coloneqq \frac{\log H(f(x))}{\log H(x)}$ is a non-zero integer, and then decide whether one of $X^r - f$ and $X^{-r} - f$ is in $\mI$.

\subsection{Case (iv): square divisor}

In this case, $\varphi$ has a square divisor. Suppose $\phi^2 \mid \varphi$ where $\deg \phi \geq 1$.
Let $x$ be a root of $\phi$, then $x \neq 0$ since $X \nmid \varphi$.
The key here is that if $\phi^2 \mid g$, then $\phi \mid g'$ where $g'$ denotes the derivative of $g$.
Indeed, writing $g = \phi^2 h$, then $g' = 2 \phi \phi' h + \phi^2 h'$ is divisible by $\phi$.

\begin{lem}\label{lem:sqfindiff}
    Let $x \neq 0$ be any root of $\phi$. 
    If $\phi^2 \mid X^a - X^b f$ where $a \neq b$, then $a - b = \frac{x f'(x)}{f(x)}$.
\end{lem}
\begin{proof}
    If $a > b$ then $\phi^2 \mid X^{a-b} - f$.
    Taking the derivative of $X^{a-b} - f$ yields 
    $
    \phi \mid (a-b) X^{a-b-1} - f'
    $.
    Since $\phi(x) = 0$ this yields $(a - b) x^{a-b-1} = f'(x)$.
    On the other hand, $\phi^2 \mid X^a - X^b f$ yields $x^{a-b} = f(x)$.
    Combining these two equations, we obtain $a - b = \frac{x f'(x)}{f(x)}$.

    If $a < b$ then $\phi^2 \mid 1 - X^{b - a} f$.
    Taking the derivative of $1 - X^{b - a} f$ yields 
    $
    \phi \mid (b - a) X^{b - a - 1} f + X^{b - a} f'
    $.
    Since $\phi(x) = 0$ this yields
    $
    (b - a) x^{b - a - 1} f(x) + x^{b - a} f'(x) = 0
    $.
    Since $\phi^2 \mid X^a - X^b f$ we have $x^a = x^b f(x)$, so $f(x) \neq 0$.
    Therefore $a - b = \frac{x f'(x)}{f(x)}$.
\end{proof}

As in the previous cases we have $X^a - X^b f \in \tI$ if and only if $X^{a - b} - f \in \mI$.
Therefore, by Lemma~\ref{lem:sqfindiff}, it suffices to decide whether $r \coloneqq \frac{x f'(x)}{f(x)}$ is a non-zero integer, and then decide whether $X^r - f \in \mI$.

\subsection{Case (v): only roots of unity}\label{subsec:casefive}
In this case, $\varphi$ is primitive, square-free, and all its roots are roots of unity.

\begin{restatable}{lem}{lemdiv}\label{lem:div}
    Let $\varphi \in \Z[X]$ be a primitive, square-free polynomial such that all its roots are roots of unity.
    Then there exists an effectively computable integer $p \geq 1$ such that $\varphi \mid X^p - 1$.
\end{restatable}
\begin{proof}
    Since $\varphi$ is square-free, it has no repeated roots over the complex numbers~\cite{yun1976square}.
    Recall that roots of unity are of the form $e^{\frac{2\pi i r}{s}}, \; r, s \in \N$.
    Let $e^{\frac{2\pi i q_1}{p_1}}, \ldots, e^{\frac{2\pi i q_d}{p_d}}$ be all the roots of $\varphi$.
    Let $p \geq 1$ be a common multiplier of $p_1, \ldots, p_d$, then these roots can be written as $e^{\frac{2\pi i Q_1}{p}}, \ldots, e^{\frac{2\pi i Q_d}{p}}$ where $Q_1, \ldots, Q_d \in [0, p-1]$ are pairwise distinct.
    Therefore $\varphi$ divides $X^p - 1 = \prod_{q = 0}^{p-1} (X - e^{\frac{2\pi i q}{p}})$ in the ring $\Q[X]$.
    Hence $\varphi$ divides $c(X^p - 1)$ in the ring $\Z[X]$ for some $c \in \Z$.
    Without loss of generality suppose $\varphi \neq 1$.
    Since $\varphi$ is primitive, it does not divide $c$, so it must divide $X^p - 1$ in the ring $\Z[X]$.
\end{proof}

Let $p \geq 1$ be such that
$
\varphi \mid X^p - 1
$.
Write
$
\tI = \varphi \cdot \tJ
$
where $\tJ$ is an ideal of $\Z[X]$ with $\gcd(\tJ) = 1$.
In particular, the generators of $\tJ$ are $\frac{g_1}{\varphi}, \ldots, \frac{g_m}{\varphi}$.
We apply Lemma~\ref{lem:idealstr} for $\tJ$.

If $\tJ = \Z[X]$, then $\tI$ is simply the ideal generated by $\varphi$.
Then $X^a - X^b f \in \tI$ if and only if $\varphi \mid X^a - X^b f$.
Since $\varphi \mid X^p - 1$, there exist $a \neq b$ such that $\varphi \mid X^a - X^b f$, if and only if there exist $a', b' \in [0, p-1]$ (not necessarily distinct), such that $\varphi \mid X^{a'} - X^{b'} f$.

If $\tJ \neq \Z[X]$, then by Lemma~\ref{lem:idealstr}, $\tJ$ contains an integer $c \geq 2$, as well as a monic polynomial $g$ of degree at least one.
Similar to Case~(i), consider the equivalent class of the elements $\frac{X^p - 1}{\varphi}, \frac{X^{2p} - 1}{\varphi}, \ldots,$ in the quotient $\Z[X]/(\Z[X] \cdot c + \Z[X] \cdot g)$.
Since $\Z[X]/(\Z[X] \cdot c + \Z[X] \cdot g)$ is finite by Lemma~\ref{lem:finring}, there exist $0 \leq q' < q$ such that 
\begin{equation}\label{eq:diffin}
\frac{X^{qp} - 1}{\varphi} - \frac{X^{q'p} - 1}{\varphi} \in \Z[X] \cdot c + \Z[X] \cdot g \subseteq \tJ.
\end{equation}
and we have the following result which is analogous to Lemma~\ref{lem:gcdone}.
\begin{restatable}{lem}{lemgcdcyc}\label{lem:gcdcyc}
    Suppose $\tJ$ contains an integer $c \geq 2$, as well as a monic polynomial $g$ of degree at least one.
    Then $\tI = \varphi \cdot \tJ$ contains an element of the form $X^a - X^b f$, $a, b \in \N, a \neq b$, if and only if $\tI$ contains an element of the form $X^{a'} - X^{b'} f$, $a', b' \in [0, pq-1]$.
\end{restatable}
Since there are only finitely many integers in $[0, pq-1]$, one can decide whether $\tI$ contains an element of the form $X^{a'} - X^{b'} f$, $a', b' \in [0, pq-1]$ by enumerating all such $a', b'$.

\bibliography{intermeta}

\begin{thebibliography}{10}

\bibitem{babai2010coset}
L{\'a}szl{\'o} Babai.
\newblock Coset intersection in moderately exponential time.
\newblock {\em Chicago J. Theoret. Comp. Sci}, 88:122--123, 2010.

\bibitem{babai1996multiplicative}
L{\'a}szl{\'o} Babai, Robert Beals, Jin-yi Cai, G{\'a}bor Ivanyos, and
  Eugene~M. Luks.
\newblock Multiplicative equations over commuting matrices.
\newblock In {\em Proceedings of the Seventh Annual ACM-SIAM Symposium on
  Discrete Algorithms}, pages 498--507, 1996.

\bibitem{baginski1999isomorphism}
Czes{\l}aw Bagi{\'n}ski.
\newblock On the isomorphism problem for modular group algebras of elementary
  abelian-by-cyclic p-groups.
\newblock {\em Colloquium Mathematicae}, 82(1):125--136, 1999.

\bibitem{baumslag1981computable}
Gilbert Baumslag, Frank~B. Cannonito, and Charles~F. Miller~III.
\newblock Computable algebra and group embeddings.
\newblock {\em Journal of Algebra}, 69(1):186--212, 1981.

\bibitem{baumslag1994algorithmic}
Gilbert Baumslag, Frank~B. Cannonito, and Derek~J.S. Robinson.
\newblock The algorithmic theory of finitely generated metabelian groups.
\newblock {\em Transactions of the American Mathematical Society},
  344(2):629--648, 1994.

\bibitem{baumslag2010subgroups}
Gilbert Baumslag, Charles~F. Miller~III, and Gretchen Ostheimer.
\newblock Subgroups of free metabelian groups.
\newblock {\em Groups, Geometry, and Dynamics}, 4(4):657--679, 2010.

\bibitem{bogopolski2010orbit}
Oleg Bogopolski, Armando Martino, and Enric Ventura.
\newblock Orbit decidability and the conjugacy problem for some extensions of
  groups.
\newblock {\em Transactions of the American Mathematical Society},
  362(4):2003--2036, 2010.

\bibitem{boler1976conjugacy}
James Boler.
\newblock Conjugacy in abelian-by-cyclic groups.
\newblock {\em Proceedings of the American Mathematical Society}, 55(1):17--21,
  1976.

\bibitem{DBLP:conf/icalp/CadilhacCZ20}
Micha{\"{e}}l Cadilhac, Dmitry Chistikov, and Georg Zetzsche.
\newblock Rational subsets of baumslag-solitar groups.
\newblock In {\em 47th International Colloquium on Automata, Languages, and
  Programming, {ICALP}}, volume 168 of {\em LIPIcs}, pages 116:1--116:16.
  Schloss Dagstuhl - Leibniz-Zentrum f{\"{u}}r Informatik, 2020.

\bibitem{delgado2018intersection}
Jordi Delgado, Enric Ventura, and Alexander Zakharov.
\newblock Intersection problem for droms raags.
\newblock {\em International Journal of Algebra and Computation},
  28(07):1129--1162, 2018.

\bibitem{farb2000asymptotic}
Benson Farb and Lee Mosher.
\newblock On the asymptotic geometry of abelian-by-cyclic groups.
\newblock {\em Acta Mathematica}, 184(2):145--202, 2000.
\newblock \href {https://doi.org/10.1007/BF02392628}
  {\path{doi:10.1007/BF02392628}}.

\bibitem{hall1954finiteness}
P.~Hall.
\newblock Finiteness conditions for soluble groups.
\newblock {\em Proceedings of the London Mathematical Society},
  s3-4(1):419--436, 01 1954.
\newblock \href {https://doi.org/10.1112/plms/s3-4.1.419}
  {\path{doi:10.1112/plms/s3-4.1.419}}.

\bibitem{howson1954intersection}
A.~G. Howson.
\newblock {On the Intersection of Finitely Generated Free Groups}.
\newblock {\em Journal of the London Mathematical Society}, s1-29(4):428--434,
  10 1954.
\newblock \href {https://doi.org/10.1112/jlms/s1-29.4.428}
  {\path{doi:10.1112/jlms/s1-29.4.428}}.

\bibitem{199499}
Benjamin~Steinberg (https://mathoverflow.net/users/15934/benjamin steinberg).
\newblock Decision problem on triviality of intersection of two subgroups.
\newblock MathOverflow.
\newblock URL:https://mathoverflow.net/q/199499 (version: 2015-03-09).
\newblock URL: \url{https://mathoverflow.net/q/199499}.

\bibitem{hurtado2021global}
Sebastian Hurtado and Jinxin Xue.
\newblock Global rigidity of some abelian-by-cyclic group actions on {T}2.
\newblock {\em Geometry \& Topology}, 25(6):3133--3178, 2021.

\bibitem{KAPOVICH2002608}
Ilya Kapovich and Alexei Myasnikov.
\newblock Stallings foldings and subgroups of free groups.
\newblock {\em Journal of Algebra}, 248(2):608--668, 2002.

\bibitem{lohrey2015rational}
Markus Lohrey, Benjamin Steinberg, and Georg Zetzsche.
\newblock Rational subsets and submonoids of wreath products.
\newblock {\em Information and Computation}, 243:191--204, 2015.

\bibitem{luks1982isomorphism}
Eugene~M. Luks.
\newblock Isomorphism of graphs of bounded valence can be tested in polynomial
  time.
\newblock {\em Journal of computer and system sciences}, 25(1):42--65, 1982.

\bibitem{macdonald2019low}
Jeremy Macdonald, Alexei Miasnikov, and Denis Ovchinnikov.
\newblock Low-complexity computations for nilpotent subgroup problems.
\newblock {\em International Journal of Algebra and Computation},
  29(04):639--661, 2019.

\bibitem{mikhailova1966occurrence}
K.~A. Mikhailova.
\newblock The occurrence problem for direct products of groups.
\newblock {\em Matematicheskii Sbornik}, 112(2):241--251, 1966.

\bibitem{Moldavanskii1968IntersectionOF}
D.~I. Moldavanskii.
\newblock Intersection of finitely generated subgroups.
\newblock {\em Siberian Mathematical Journal}, 9:1066--1069, 1968.
\newblock URL: \url{https://api.semanticscholar.org/CorpusID:123410691}.

\bibitem{noskov1982conjugacy}
Gennady~Andreevich Noskov.
\newblock Conjugacy problem in metabelian groups.
\newblock {\em Mathematical notes of the Academy of Sciences of the USSR},
  31:252--258, 1982.

\bibitem{pittet2003random}
Christophe Pittet and Laurent Saloff-Coste.
\newblock Random walks on abelian-by-cyclic groups.
\newblock {\em Proceedings of the American Mathematical Society},
  131(4):1071--1079, 2003.

\bibitem{potapov2019vector}
Igor Potapov and Pavel Semukhin.
\newblock Vector and scalar reachability problems in
  $\operatorname{SL}_2(\mathbb{Z})$.
\newblock {\em Journal of Computer and System Sciences}, 100:30--43, 2019.

\bibitem{romanovskii1974some}
N.~S. Romanovskii.
\newblock Some algorithmic problems for solvable groups.
\newblock {\em Algebra and Logic}, 13(1):13--16, 1974.

\bibitem{sharpe1987rings}
David Sharpe.
\newblock {\em Rings and factorization}.
\newblock Cambridge University Press, 1987.

\bibitem{stallings1991foldings}
John~R. Stallings.
\newblock Foldings of {G}-trees.
\newblock In {\em Arboreal Group Theory: Proceedings of a Workshop Held
  September 13--16, 1988}, pages 355--368. Springer, 1991.

\bibitem{szekeres1952canonical}
George Szekeres.
\newblock A canonical basis for the ideals of a polynomial domain.
\newblock {\em The American Mathematical Monthly}, 59(6):379--386, 1952.

\bibitem{waldschmidt2013diophantine}
Michel Waldschmidt.
\newblock {\em Diophantine approximation on linear algebraic groups:
  transcendence properties of the exponential function in several variables},
  volume 326.
\newblock Springer Science \& Business Media, 2013.

\bibitem{yun1976square}
David Y.~Y. Yun.
\newblock On square-free decomposition algorithms.
\newblock In {\em Proceedings of the third ACM symposium on Symbolic and
  algebraic computation}, pages 26--35, 1976.

\end{thebibliography}

\newpage

\appendix
\section{Omitted proofs}\label{app:proof}

\lemchangebase*
\begin{proof}
    Every element $\bff \in \Z[X^{\pm}]^D$ can be uniquely written as $\bff = \bff_0 + X \cdot \bff_1 + \cdots + X^{d-1} \cdot \bff_{d-1}$ where $\bf_0, \bf_1, \ldots, \bf_{d-1}$ are in $\Z[X^{\pm d}]$.
    This gives an effective isomorphism $\varphi \colon \Z[X^{\pm}]^D \rightarrow \Z[X^{\pm d}]^{Dd}$.
    Write $\mA = M/N$ where $M, N \subseteq \Z[X^{\pm}]^D$.
    The generators of $\varphi(M)$ can be obtained by simply applying $\varphi$ to the generators of $M$, similarly for $\varphi(N)$.
    Hence $\mA = \varphi(M)/\varphi(N)$ is a finite presentation of $\mA$ as a $\Z[X^{\pm d}]$-module.
\end{proof}

\lemstruct*
\begin{proof}
    (i) is obvious.
    For (ii), suppose $z_1, \ldots, z_K$ are not all zero and let $d \in \N$ be their greatest common divisor.
    Let $n_1, \ldots, n_K \in \Z$ be such that $n_1 z_1 + \cdots + n_K z_K = d$, then $g \coloneqq g_1^{n_1} \cdots g_K^{n_K}$ is of the form $(\bb, d), \bb \in \mA$.
    Then for any $\ba \in \gen{\mG} \cap \mA$, we have $\gen{\mG} \cap \mA \ni g^{-1} \ba g = (- X^{-d} \cdot \bb, -d) (\ba, 0) (\bb, d) = (X^d \cdot \ba, 0) = X^d \cdot \ba$.
    Similarly, $\gen{\mG} \cap \mA \ni g \ba g^{-1} = X^{-d} \cdot \ba$.
    Therefore, $\gen{\mG} \cap \mA$ is a $\Z[X^{\pm d}]$-module.

    On one hand, the elements in $S$ are obviously in $\gen{\mG} \cap \mA$, so the $\Z[X^{\pm d}]$-module generated by $S$ is a submodule of $\gen{\mG} \cap \mA$.
    On the other hand, we show that the quotient $(\gen{\mG} \cap \mA)/S$ is trivial.
    Notice that since $(g_i g_j g_i^{-1} g_j^{-1})^{-1} = g_j g_i g_j^{-1} g_i^{-1}$, we have $\left\{g_i g_j g_i^{-1} g_j^{-1} \;\middle|\; i, j \in [1, K] \right\} \subseteq S$.
    Therefore, the quotient by $S$ allows one to permute elements in any product $g_{i_1}^{\epsilon_1} \cdots g_{i_n}^{\epsilon_n} \in \gen{\mG} \cap \mA$ without changing their class in $(\gen{\mG} \cap \mA)/S$.
    More precisely, for $g, g' \in \gen{\mG}$ and $i, j \in [1, K]$, if $g g_i g_j g' \in \mA$ then we have $g g_i g_j g' + S = g g_j g_i g' + S$.
    Indeed, we have $g g_i g_j g' = g (g_i g_j g_i^{-1} g_j^{-1}) g_j g_i g' = g (g_i g_j g_i^{-1} g_j^{-1}) g^{-1} + g g_j g_i g'$ and $g (g_i g_j g_i^{-1} g_j^{-1}) g^{-1}$ is in the $\Z[X^{\pm d}]$-module generated by $S$.
    For every product $g_{i_1}^{\epsilon_1} \cdots g_{i_n}^{\epsilon_n} \in \gen{\mG} \cap \mA$, we must have 
    \[
    \left(\sum_{j \in [1, n], i_j = 1} \epsilon_j, \ldots, \sum_{j \in [1, n], i_j = K} \epsilon_j\right) \in \Lambda
    \]
    by looking at the second component.
    Since $(s_{11}, \ldots, s_{1K}), \ldots, (s_{T1}, \ldots, s_{TK})$ be are the generators for $\Lambda$, by permuting the elements in the product we can rewrite $g_{i_1}^{\epsilon_1} \cdots g_{i_n}^{\epsilon_n}$ as $\left(g_1^{s_{1 1}} \cdots g_K^{s_{1 K}}\right)^{j_1} \cdots \left(g_1^{s_{T 1}} \cdots g_K^{s_{T K}}\right)^{j_T}$ where $j_1, \ldots, j_T \in \Z$.
    Therefore, $g_{i_1}^{\epsilon_1} \cdots g_{i_n}^{\epsilon_n}$ is in the $\Z[X^{\pm d}]$-module generated by $S$.

    For (iii), $\ba \in \mA$ be any element such that $(\ba, d) \in \gen{\mG}$.
    Since $d \in \N$ is the greatest common divisor for $z_1, \ldots, z_K$, every element $g$ of $\gen{\mG}$ must be of the form $(\bc, md), \; \bc \in \mA, m \in \Z$.
    Then $g \cdot (\ba, d)^{-m} \in \gen{\mG} \cap \mA$. Let $(\bb, 0) \coloneqq g \cdot (\ba, d)^{-m}$, then $\ba \in \gen{\mG} \cap \mA$ and $g = (\bb, 0) \cdot (\ba, d)^m$.
\end{proof}

\propCtwo*
\begin{proof}
We can without loss of generality suppose $\gen{\mH} \subseteq \mA$ and $\gen{\mG} \not\subset \mA$.
Otherwise notice that $\gen{\mG} \cap h \gen{\mH} = \emptyset$ if and only if $h^{-1} \gen{\mG} \cap \gen{\mH} = \emptyset$, so we can exchange the role of $\gen{\mG}$ and $\gen{\mH}$.

By Lemma~\ref{lem:struct}, suppose $\gen{\mG}$ is generated by the element $(\ba_{\mG}, d_{\mG})$ and the $\Z[X^{\pm d_{\mG}}]$-module $\gen{\mG} \cap \mA$, and $\gen{\mH}$ is generated by the elements $h_1 = (\ba'_1, 0), \ldots, h_M = (\ba'_M, 0)$.
Recall that the generators of the $\Z[X^{\pm d_{\mG}}]$-module $\gen{\mG} \cap \mA$ can be effectively computed.
Also, by Lemma~\ref{lem:changebase}, we can consider $\mA$ as a finitely presented $\Z[X^{\pm d_{\mG}}]$-module instead of a $\Z[X^{\pm}]$-module, and suppose the generators of $\gen{\mG} \cap \mA$ as well as $\ba'_1, \ldots, \ba'_M$ are given as elements of the $\Z[X^{\pm d_{\mG}}]$-module.

\textbf{Subgroup Intersection.}
In this case, $\gen{\mG} \cap \gen{\mH} = \{e\}$ if and only if every solution of $z_1 \cdot \ba'_1 + \cdots + z_M \cdot \ba'_M \in \gen{\mG} \cap \mA, \; z_1, \ldots, z_M \in \Z$ is also a solution of $z_1 \cdot \ba'_1 + \cdots + z_M \cdot \ba'_M = \bzer$.

Let $\mM$ denote the $\Z[X^{\pm d_{\mG}}]$-module
\begin{equation}\label{eq:M2s}
\mM \coloneqq \Big\{ (z_1, \ldots, z_M) \in \Z[X^{\pm d_{\mG}}]^{M} \;\Big|\;
z_1 \cdot \ba'_1 + \cdots + z_M \cdot \ba'_M \in \gen{\mG} \cap \mA \Big\}, \\
\end{equation}
and $\mZ$ denote the $\Z[X^{\pm d_{\mG}}]$-module
\begin{equation}\label{eq:Z2s}
\mZ \coloneqq \Big\{ (z_1, \ldots, z_M) \in \Z[X^{\pm d_{\mG}}]^{M} \;\Big|\;
z_1 \cdot \ba'_1 + \cdots + z_M \cdot \ba'_M = \bzer \Big\}.
\end{equation}

Then we have $\gen{\mG} \cap \gen{\mH} = \{e\}$ if and only if $\mM \cap \Z^M \neq \mZ \cap \Z^M$.

Note that $z_1 \cdot \ba'_1 + \cdots + z_M \cdot \ba'_M \in \gen{\mG} \cap \mA$ can be rewritten as a linear equation 
\begin{equation}\label{eq:equiveq}
    z_1 \cdot \ba'_1 + \cdots + z_M \cdot \ba'_M + x_1 \cdot \bg_1 + \cdots + x_m \cdot \bg_m = \bzer, \quad x_1, \ldots, x_m \in \Z[X^{\pm d_{\mG}}],
\end{equation}
where $\bg_1, \ldots, \bg_m$ are the generators of $\gen{\mG} \cap \mA$.
Therefore the generators of $\mM$ can be computed by projecting the solution set of \eqref{eq:equiveq} to the coordinates $(z_1, \ldots, z_M)$.
Similar to the previous case, Subgroup Intersection can be decided by computing the generators of $\mM \cap \Z^M$ and $\mZ \cap \Z^M$ using Lemma~\ref{lem:classicdec} and \ref{lem:decinterZ}, and deciding equality using linear algebra over $\Z$.

\textbf{Coset Intersection.}
Let $h = (\ba_h, z_h)$.
If $d_{\mG} \nmid z_h$ then $\gen{\mG} \cap h \gen{\mH} = \emptyset$.
Therefore we only need to consider the case where $z_h = z d_{\mG}$ for some $z \in \Z$.
Then $\gen{\mG} \cap h \gen{\mH} \neq \emptyset$ if and only if the equation
$
(\bb, 0) \cdot (\ba_{\mG}, d_{\mG})^z = (\ba_h, z_h) \cdot (\bc, 0)
$
has solutions $\bb \in \gen{\mG} \cap \mA,\; \bc \in \sum_{i = 1}^M \Z \cdot \ba'_i$.
Direct computation shows this is equivalent to
\[
    X^{z_h} \cdot \bc + \left(\ba_h - \frac{X^{z d_{\mG}} - 1}{X^{d_{\mG}} - 1} \cdot \ba_{\mG} \right) = \bb.
\]

Let $\mM'$ denote the $\Z[X^{\pm d_{\mG}}]$-module
\begin{multline}\label{eq:M2c}
\mM' \coloneqq \Bigg\{ (z_1, \ldots, z_M, z) \in \Z[X^{\pm d_{\mG}}]^{M+1} \;\Bigg|\; \\
X^{z_h} \cdot \left(z_1 \cdot \ba'_1 + \cdots + z_M \cdot \ba'_M\right) + z \cdot \left(\ba_h - \frac{X^{z_h} - 1}{X^{d_{\mG}} - 1} \cdot \ba_{\mG} \right) \in \gen{\mG} \cap \mA \Bigg\}. 
\end{multline}
Then $\gen{\mG} \cap h \gen{\mH} = \emptyset$ if and only if $d_{\mG} \nmid z_h$ and $\big(\mM' \cap \Z^{M+1}\big) \cap \big(\Z^{M} \times \{1\}\big) = \emptyset$.
Similarly, Coset Intersection is decidable in this case.
\end{proof}

\lemcosettoeq*
\begin{proof}
    Suppose the $\gen{\mG} \cap h \gen{\mH}$ is non-empty. Let $(\ba, z') \in \gen{\mG} \cap h \gen{\mH}$, then $d_{\mG} \mid z'$ and $d_{\mH} \mid (z' - z_h)$.
    Hence $z' = z_{\mG} + zd = z_{\mH} + zd + z_h$ for some $z \in \Z$.
    Since $(\ba, z') \in \gen{\mG} \cap h \gen{\mH}$, the Equation~\eqref{eq:Cosetsys} has solution with $m d_{\mG} = z'$, meaning
    \begin{multline*}
        \frac{X^{z d} \cdot \ba'_{\mG, \mH} - \ba''_{\mG, \mH}}{X^d - 1} =
        \frac{X^{z_{\mG} + zd}}{X^{d_{\mG}} - 1} \cdot \ba_{\mG} - \frac{X^{z_{\mH} + zd}}{X^{d_{\mH}} - 1} \cdot \ba_{\mH} - \frac{1}{X^{d_{\mG}} - 1} \cdot \ba_{\mG} + \frac{1}{X^{d_{\mH}} - 1} \cdot \ba_{\mH} - \ba_h \\
        = \frac{X^{z_{\mG} + zd} - 1}{X^{d_{\mG}} - 1} \cdot \ba_{\mG} - X^{z_h} \cdot \frac{X^{z_{\mH} + zd} - 1}{X^{d_{\mH}} - 1} \cdot \ba_{\mH} - \ba_h = X^{z_h} \cdot \bc - \bb \in \mM'.
    \end{multline*}
    Therefore \eqref{eq:Coseteqz} is satisfied.

    On the other hand, suppose Equation~\eqref{eq:Coseteqz} is satisfied.
    Then we have
    \[
    \frac{X^{z_{\mG} + zd} - 1}{X^{d_{\mG}} - 1} \cdot \ba_{\mG} - X^{z_h} \cdot \frac{X^{z_{\mH} + zd} - 1}{X^{d_{\mH}} - 1} \cdot \ba_{\mH} - \ba_h = \frac{X^{z d} \cdot \ba'_{\mG, \mH} - \ba''_{\mG, \mH}}{X^d - 1} \in \mM',
    \]
    so it can be written as $X^{z_h} \cdot \bc - \bb$ for some $\bb \in \gen{\mG} \cap \mA, \; \bc \in \gen{\mH} \cap \mA$.
    Hence the system~\eqref{eq:Cosetsys} has solutions $\bb \in \gen{\mG} \cap \mA,\; \bc \in \gen{\mH} \cap \mA,\; m = \frac{z_{\mG} + z d}{d_{\mG}},\; n = \frac{z_{\mH} + z d}{d_{\mH}}$.
\end{proof}

\lemctoi*
\begin{proof}
    Since $f_0 \cdot \ba'_{\mG, \mH} - \ba''_{\mG, \mH} \in (X^d - 1) \cdot \mM'$, we have $f \cdot \ba'_{\mG, \mH} - \ba''_{\mG, \mH} \in (X^d - 1) \cdot \mM'$ if and only if $(f - f_0) \cdot \ba'_{\mG, \mH} \in (X^d - 1) \cdot \mM'$.
    This is equivalent to $f - f_0 \in \mI'$, in other words $f \in f_0 + \mI'$.
\end{proof}

\lemlaurenttoreg*
\begin{proof}
    Suppose $g \in \mI$, then $g = f_1 \cdot g_1 + \cdots + f_m \cdot g_m$ for some $f_1, \ldots, f_m \in \Z[X^{\pm}]$. Let $c \in \N$ be large enough so that $X^c f_1, \ldots, X^c f_m$ are in $\Z[X]$, then $X^c g = X^c f_1 \cdot g_1 + \cdots + X^c f_m \cdot g_m \in \tI$.
    On the other hand, suppose $X^c g \in \tI$ for some $c \in \N$, then $X^c g = F_1 \cdot g_1 + \cdots + F_m \cdot g_m$ for some $F_1, \ldots, F_m \in \Z[X]$.
    Then $g = X^{-c} F_1 \cdot g_1 + \cdots + X^{-c} F_m \cdot g_m \in \mI$.
\end{proof}

\lemgcdcyc*
\begin{proof}
    By Equation~\eqref{eq:diffin}, we have
    $
    \frac{X^{qp} - X^{q'p}}{\varphi} \in \tJ
    $,
    so $X^{qp} - X^{q'p} \in \varphi \cdot \tJ = \tI$.
    Therefore, every $X^r, \; r \in \N$ is equivalent modulo $\tI$ to $X^{r'}$ for some $r' \in [0, pq-1]$.
    Therefore similar to Lemma~\ref{lem:gcdone}, $\tI$ contains an element of the form $X^a - X^b f$, $a, b \in \N, a \neq b$, if and only if $\tI$ contains an element of the form $X^{a'} - X^{b'} f$, $a', b' \in [0, pq-1]$.
\end{proof}

\newpage

\section{Summary of algorithms}\label{app:alg}

\begin{algorithm}[!ht]
\caption{Algorithm for Subgroup Intersection}
\label{alg:groupinter}
\begin{description} 
\item[Input:]
a finite presentation of the $\Z[X^{\pm}]$-module $\mA$, two finite sets of elements $\mG = \{(\ba_1, z_1), \ldots, (\ba_K, z_K)\},\; \mH = \{(\ba'_1, z'_1), \ldots, (\ba'_M, z'_M)\}$ in the group $\mA \rtimes \Z$.
\item[Output:] \textbf{True} (when $\gen{\mG} \cap \gen{\mH} = \{e\}$) or \textbf{False} (when $\gen{\mG} \cap \gen{\mH} \neq \{e\}$).
\end{description}
\begin{enumerate}[1.]
    \item \textbf{If $z_1, \ldots, z_K$ and $z'_1, \ldots, z'_M$ are all zero.}
    
    Compute generators of the modules $\mM$ and $\mZ$ defined in Equations~\eqref{eq:M1s} and \eqref{eq:Z1s}.
    
    Decide whether 
    $
    \mM \cap \Z^{K+M} = (\mM \cap \mZ) \cap \Z^{K+M}
    $
    using Lemma~\ref{lem:classicdec} and \ref{lem:decinterZ}.
    If yes, return \textbf{True}, otherwise return \textbf{False}.
    
    \item \textbf{If one of the sets $\{z_1, \ldots, z_K\}$ and $\{z'_1, \ldots, z'_M\}$ is all zero.}
    
    Without loss of generality suppose $z'_1 = \cdots = z'_M = 0$, otherwise swap the sets $\mG, \mH$.
    \begin{enumerate}[(i)]
        \item Compute $d_{\mG} \coloneqq \gcd(z_1, \ldots, z_K)$, and compute the generators of the $\Z[X^{\pm d_{\mG}}]$-module $\gen{\mG} \cap \mA$ using Lemma~\ref{lem:struct}.
        \item Compute generators of the modules $\mM$ and $\mZ$ in Equations~\eqref{eq:M2s} and \eqref{eq:Z2s}.
        
        Decide whether 
        $
        \mM \cap \Z^M \neq \mZ \cap \Z^M
        $
        using Lemma~\ref{lem:classicdec} and \ref{lem:decinterZ}.
        If yes, return \textbf{True}, otherwise return \textbf{False}.
    \end{enumerate}
    \item \textbf{If none of the sets $\{z_1, \ldots, z_K\}$ and $\{z'_1, \ldots, z'_M\}$ is all zero.}
    \vspace{-\baselineskip}
    \begin{enumerate}[(i)]
        \item Compute $d_{\mG} \coloneqq \gcd(z_1, \ldots, z_K), d_{\mH} \coloneqq \gcd(z'_1, \ldots, z'_M), d \coloneqq \gcd(d_{\mG}, d_{\mH})$.
        \item Compute generators of the $\Z[X^{\pm d_{\mG}}]$-module $\gen{\mG} \cap \mA$ and the $\Z[X^{\pm d_{\mH}}]$-module $\gen{\mH} \cap \mA$ using Lemma~\ref{lem:struct}.
        Compute their respective generators $S_{\mG}, S_{\mH}$ as $\Z[X^{\pm d}]$-modules using Lemma~\ref{lem:changebase}.
        Let $\mM$ be the $\Z[X^{\pm d}]$-module generated by $S_{\mG} \cup S_{\mH}$.
        \item Decide whether $(\gen{\mG} \cap \mA) \cap (\gen{\mH} \cap \mA) = \{\bzer\}$ using Lemma~\ref{lem:classicdec}(iii).
        If yes, continue, otherwise return \textbf{False}.
        \item Compute the generators of the ideal $\mI \subseteq \Z[X^{\pm d}]$ defined in Equation~\eqref{eq:defI} using the generators of $\mM$ and Lemma~\ref{lem:classicdec}(ii).
        \item Using Algorithm~\ref{alg:smm}, decide whether $\mI$ contains an element $X^{zd} - 1$ for some $z \in \Z \setminus \{0\}$.
        If yes, return \textbf{False}, otherwise return \textbf{True}.
    \end{enumerate}
\end{enumerate}
\end{algorithm}

\begin{algorithm}[!ht]
\caption{Algorithm for Coset Intersection}
\label{alg:cosetinter}
\begin{description} 
\item[Input:]
a finite presentation of the $\Z[X^{\pm}]$-module $\mA$, two finite sets of elements $\mG = \{(\ba_1, z_1), \ldots, (\ba_K, z_K)\},\; \mH = \{(\ba'_1, z'_1), \ldots, (\ba'_M, z'_M)\}$ in the group $\mA \rtimes \Z$, an element $h = (\ba_h, z_h)$.
\item[Output:] \textbf{True} (when $\gen{\mG} \cap h \gen{\mH} = \emptyset$) or \textbf{False} (when $\gen{\mG} \cap h \gen{\mH} \neq \emptyset$).
\end{description}
\begin{enumerate}[1.]
    \item \textbf{If $z_1, \ldots, z_K$ and $z'_1, \ldots, z'_M$ are all zero.}
    
    Compute generators of the module $\mM'$ defined in Equation~\eqref{eq:M1c}.
    
    Decide whether 
    $
    \big(\mM' \cap \Z^{K+M+1}\big) \cap \big(\Z^{K+M} \times \{1\}\big) = \emptyset
    $
    using Lemma~\ref{lem:classicdec} and \ref{lem:decinterZ}.
    If yes, return \textbf{True}, otherwise return \textbf{False}.
    
    \item \textbf{If one of the sets $\{z_1, \ldots, z_K\}$ and $\{z'_1, \ldots, z'_M\}$ is all zero.}
    
    Without loss of generality suppose $z'_1 = \cdots = z'_M = 0$, otherwise swap the sets $\mG, \mH$ and replace $h$ with $h^{-1}$.
    \begin{enumerate}[(i)]
        \item Compute $d_{\mG} \coloneqq \gcd(z_1, \ldots, z_K)$, and compute generators of the $\Z[X^{\pm d_{\mG}}]$-module $\gen{\mG} \cap \mA$ using Lemma~\ref{lem:struct}.
        \item If $d_{\mG} \mid z_h$, continue, otherwise return \textbf{False}.
        \item Compute generators of the modules $\mM'$ defined in Equation~\eqref{eq:M2c}.
        
        Decide whether 
        $
        \big(\mM' \cap \Z^{M+1}\big) \cap \big(\Z^{M} \times \{1\}\big) = \emptyset
        $
        using Lemma~\ref{lem:classicdec} and \ref{lem:decinterZ}.
        If yes, return \textbf{True}, otherwise return \textbf{False}.
    \end{enumerate}
    \item \textbf{If none of the sets $\{z_1, \ldots, z_K\}$ and $\{z'_1, \ldots, z'_M\}$ is all zero.}
    \vspace{-\baselineskip}
    \begin{enumerate}[(i)]
        \item Compute $d_{\mG} \coloneqq \gcd(z_1, \ldots, z_K), d_{\mH} \coloneqq \gcd(z'_1, \ldots, z'_M), d \coloneqq \gcd(d_{\mG}, d_{\mH})$.
        \item Decide whether the equation $m d_{\mG} = n d_{\mH} + z_h$ has solutions $(m, n) \in \Z^2$.
        If yes, take any solution $(m, n)$ and let $z_{\mG} \coloneqq m d_{\mG},\; z_{\mH} \coloneqq n d_{\mH}$; otherwise return \textbf{True}.
        \item Compute generators of the $\Z[X^{\pm d_{\mG}}]$-module $\gen{\mG} \cap \mA$ and the $\Z[X^{\pm d_{\mH}}]$-module $\gen{\mH} \cap \mA$ using Lemma~\ref{lem:struct}.
        Compute their respective generators $S_{\mG}, S_{\mH}$ as $\Z[X^{\pm d}]$-modules using Lemma~\ref{lem:changebase}.
        Let $\mM'$ be the $\Z[X^{\pm d}]$-module generated by $S_{\mG} \cup (X^{z_h} \cdot S_{\mH})$.
        \item Using Lemma~\ref{lem:classicdec}(i), decide whether Equation~\eqref{eq:faainM} has a solution $f \in \Z[X^{\pm d}]$. \\
        If yes, compute (by enumeration) such a solution $f_0$; otherwise return \textbf{True}.
        \item Compute the generators of the ideal $\mI' \subseteq \Z[X^{\pm d}]$ defined in Equation~\eqref{eq:defIp} using the generators of $\mM'$ and Lemma~\ref{lem:classicdec}(ii).
        \item Using Lemma~\ref{lem:classicdec}(i), decide whether $1 - f_0 \in \mI'$.
        If yes, return \textbf{False}, otherwise continue.
        \item Using Algorithm~\ref{alg:smm}, decide whether $\mI'$ contains an element $X^{zd} - f_0$ for some $z \in \Z \setminus \{0\}$.
        If yes, return \textbf{False}, otherwise return \textbf{True}.
    \end{enumerate}
\end{enumerate}
\end{algorithm}

\begin{algorithm}[!ht]
\caption{Algorithm for Shifted Monomial Membership}
\label{alg:smm}
\begin{description} 
\item[Input:] The generators $g_1, \ldots, g_m$ of an ideal $\mI \subseteq \Z[X^{\pm}]$ and an element $f \in \Z[X^{\pm}]$.

\item[Output:] \textbf{True} (when there exists $z \in \Z \setminus \{0\}$ such that $X^{z} - f \in \mI$) or \textbf{False}.
\end{description}
\begin{enumerate}[1.]
    \item Multiply each $g_i, i = 1, \ldots, m,$ by a suitable power of $X$ so that $g_i \in \Z[X]$ and $X \nmid g_i$.
    
    Denote by $\tI$ the ideal of $\Z[X]$ generated by $g_1, \ldots, g_m$.
    \item Compute $\varphi \coloneqq \gcd(g_1, \ldots, g_m)$.
    \item \textbf{If $\varphi = 1$.}
    \begin{enumerate}[(i)]
        \item If $\tI = \Z[X]$ (equivalent to $1 \in \tI$), return \textbf{True}; otherwise compute $c, g \in \tI$ defined in Lemma~\ref{lem:idealstr}.
        \item Enumerate all pairs $(p, q) \in \N^2,\; p < q,$ and test whether $X^q - X^p \in (\Z[X] \cdot g + \Z[X] \cdot c)$. 
        Stop when we find such a pair $p, q$.
        \item For all pairs of integers $a', b' \in [0, q-1]$, decide whether $X^{a'} - X^{b'} f \in \tI$.
        If such a pair exists, return \textbf{True}; otherwise return \textbf{False}.
    \end{enumerate}
    
    \item \textbf{If $\varphi$ is not primitive.}  
    
    For each integer $r \in [1, \deg f]$, decide whether $X^r - f \in \mI$.
    If such $r$ exists, return \textbf{True}; otherwise return \textbf{False}.
    
    \item \textbf{If $\varphi$ is primitive and has a root $x$ that is not a root of unity.}
    
    Let $\K$ be an algebraic number field that contains $x$ and let $H$ be the height function over $\K^*$ (Lemma~\ref{lem:height}).

    Let $r \coloneqq \frac{\log H(f(x))}{\log H(x)}$.
    If $r = 0$ or $r$ is not an integer, return \textbf{False}.
    
    Otherwise decide whether one of $X^{r} - f$ and $X^{-r} - f$ is in $\mI$.
    If yes, return \textbf{True}; otherwise return \textbf{False}.

    \item \textbf{If $\varphi$ is primitive, all roots of $\varphi$ are roots of unity, and $\varphi$ has a square divisor $\phi$.}

    Let $x$ be a root of $\phi$. Let $r \coloneqq \frac{x f'(x)}{f(x)}$.
    
    If $r$ is a non-zero integer and $X^r - f \in \mI$, return \textbf{True}; otherwise return \textbf{False}.

    \item \textbf{If $\varphi$ is primitive, all roots of $\varphi$ are roots of unity, and $\varphi$ is square-free.}
    \begin{enumerate}[(i)]
        \item Compute $p \geq 1$ such that $\varphi \mid X^p - 1$ (Lemma~\ref{lem:div}).
        \item Compute the generators $g'_i \coloneqq \frac{g_i}{\varphi}, i = 1, \ldots, m$ of the ideal $\tJ = \frac{\tI}{\varphi}$.
        \item If $\tJ = \Z[X]$ (equivalent to $1 \in \tJ$). 
        
        For each pair of integers $a', b' \in [0, p-1]$, decide whether $\varphi \mid X^{a'} - X^{b'} f$.
        If such a pair exists return \textbf{True}, otherwise return \textbf{False}.
        \item If $\tJ \neq \Z[X]$, compute $c, g \in \tJ$ defined in Lemma~\ref{lem:idealstr}.
        \begin{enumerate}[(a)]
            \item Enumerate all pairs $(q', q) \in \N^2,\; q' < q,$ and test whether $\frac{X^{pq} - 1}{\varphi} - \frac{X^{pq'} - 1}{\varphi} \in \Z[X] \cdot c + \Z[X] \cdot g$. 
            Stop when we find such a pair $q', q$.
            \item For all pairs of integers $a', b' \in [0, pq-1]$, decide whether $X^{a'} - X^{b'} f \in \tI$.
            If such a pair exists, return \textbf{True}; otherwise return \textbf{False}.
        \end{enumerate}
    \end{enumerate}
\end{enumerate}
\end{algorithm}
\end{document}